\documentclass[12pt]{article}
\usepackage{amsmath,amssymb,amsthm,color,a4wide,comment}
\usepackage{graphicx}
\definecolor{darkblue}{rgb}{0.,0.1,0.5}
\definecolor{darkgreen}{rgb}{0.,0.5,0.1}
\definecolor{darkred}{rgb}{0.5,0.,0.1}
\definecolor{darkcyan}{rgb}{0.,0.5,0.5}

\newcommand\llam{\hat{\lambda}}
\newcommand\TT{\bar{l}}
\newcommand\ttt{y}
\def\uT{u_{+0}}
\def\uTtd{u_{+0}^{\tilde{\delta}}}
\def\uTdel{u_{+0}^{\delta}}
\newcommand\nnu{{\tilde{\nu}}}
\newcommand\ttau{{\tilde{\tau}}}
\newcommand\altil{{\tilde{\alpha}}}
\newcommand\betil{{\tilde{\beta}}}
\newcommand\gamtil{{\tilde{\gamma}}}
\newcommand\coeffalpha{{\mathfrak{a}}}
\newcommand\coeffbeta{{\mathfrak{b}}}
\newcommand\dl{\underline{d\ell}}
\newcommand\dgam{\underline{d\gamma}}
\newcommand\du{\underline{du}}
\newcommand\Ftil{\tilde{F}}
\newcommand\regpar{\varepsilon}
\newcommand\ulell{\underline{l}}
\newcommand\olell{\overline{l}}
\newcommand{\neumann}[1]{{#1}}
\newcommand{\dirichlet}[1]{{#1}}
\newcommand{\impedance}[1]{{#1}}

\newtheorem{theorem}{Theorem}[section]
\newtheorem{corollary}{Corollary}[section]
\newtheorem{problem}{Problem}[section]
\newtheorem{lemma}{Lemma}[section]
\newtheorem{remark}{Remark}[section]
\input colordvi
\font\tenrm=cmr10
\font\teni=cmmi10 \skewchar\teni='177
\font\tensy=cmsy10 \skewchar\tensy='60
\font\tenex=cmex10
\font\tenit=cmti10
\font\tensl=cmsl10
\font\tenbf=cmbx10
\font\tentt=cmtt10
\font\ninerm=cmr9
\font\ninei=cmmi9 \skewchar\ninei='177
\font\ninesy=cmsy9 \skewchar\ninesy='60
\font\nineit=cmti9
\font\ninesl=cmsl9
\font\ninebf=cmbx9
\font\ninett=cmtt9
\font\eightrm=cmr8
\font\eighti=cmmi8 \skewchar\eighti='177
\font\eightsy=cmsy8 \skewchar\eightsy='60
\font\eightit=cmti8
\font\eightsl=cmsl8
\font\eightbf=cmbx8
\font\eighttt=cmtt8
\font\sevenrm=cmr7
\font\seveni=cmmi7 \skewchar\seveni='177
\font\sevensy=cmsy7 \skewchar\sevensy='60
\font\sevenbf=cmbx7
\font\sevenit=cmmi7
\font\sevensl=cmmi7
\font\seventt=cmr7
\font\sixrm=cmr6
\font\sixi=cmmi6 \skewchar\sixi='177
\font\sixsy=cmsy6 \skewchar\sixsy='60
\font\sixbf=cmbx6
\font\fiverm=cmr5
\font\fivei=cmmi5 \skewchar\fivei='177
\font\fivesy=cmsy5 \skewchar\fivesy='60
\font\fivebf=cmbx5
\def\tenpoint{\def\rm{\fam0\tenrm}%
        \textfont0=\tenrm \scriptfont0=\sevenrm \scriptscriptfont0=\fiverm
        \textfont1=\teni \scriptfont1=\seveni \scriptscriptfont1=\fivei
        \textfont2=\tensy \scriptfont2=\sevensy \scriptscriptfont2=\fivesy
        \textfont3=\tenex \scriptfont3=\tenex \scriptscriptfont3=\tenex
        \def\it{\fam\itfam\tenit}%
        \textfont\itfam=\tenit
        \def\sl{\fam\slfam\tensl}%
        \textfont\slfam=\tensl
        \def\bf{\fam\bffam\tenbf}%
        \textfont\bffam=\tenbf \scriptfont\bffam=\sevenbf
                \scriptscriptfont\bffam=\fivebf
        \def\tt{\fam\ttfam\tentt}%
        \textfont\ttfam=\tentt
        \normalbaselineskip=12pt%
        \let\sc=\eightrm        
        \setbox\strutbox=\hbox{\vrule height8.5pt depth3.5pt width0pt}%
        \normalbaselines\rm}
\def\ninepoint{\def\rm{\fam0\ninerm}%
        \textfont0=\ninerm \scriptfont0=\sixrm \scriptscriptfont0=\fiverm
        \textfont1=\ninei \scriptfont1=\sixi \scriptscriptfont1=\fivei
        \textfont2=\ninesy \scriptfont2=\sixsy \scriptscriptfont2=\fivesy
        \textfont3=\tenex \scriptfont3=\tenex \scriptscriptfont3=\tenex
        \def\it{\fam\itfam\nineit}%
        \textfont\itfam=\nineit
        \def\sl{\fam\slfam\ninesl}%
        \textfont\slfam=\ninesl
        \def\bf{\fam\bffam\ninebf}%
        \textfont\bffam=\ninebf \scriptfont\bffam=\sixbf
                \scriptscriptfont\bffam=\fivebf
        \def\tt{\fam\ttfam\ninett}%
        \textfont\ttfam=\ninett
        \normalbaselineskip=11pt%
        \let\sc=\sevenrm        
        \setbox\strutbox=\hbox{\vrule height8pt depth3pt width0pt}%
        \normalbaselines\rm}
\def\eightpoint{\def\rm{\fam0\eightrm}%
        \textfont0=\eightrm \scriptfont0=\sixrm \scriptscriptfont0=\fiverm
        \textfont1=\eighti \scriptfont1=\sixi \scriptscriptfont1=\fivei
        \textfont2=\eightsy \scriptfont2=\sixsy \scriptscriptfont2=\fivesy
        \textfont3=\tenex \scriptfont3=\tenex \scriptscriptfont3=\tenex
        \def\it{\fam\itfam\eightit}%
        \textfont\itfam=\eightit
        \def\sl{\fam\slfam\eightsl}%
        \textfont\slfam=\eightsl
        \def\bf{\fam\bffam\eightbf}%
        \textfont\bffam=\eightbf \scriptfont\bffam=\sixbf
                \scriptscriptfont\bffam=\fivebf
        \def\tt{\fam\ttfam\eighttt}%
        \textfont\ttfam=\eighttt
        \normalbaselineskip=9pt%
        \let\sc=\sixrm  
        \setbox\strutbox=\hbox{\vrule height7pt depth2pt width0pt}%
        \normalbaselines\rm}
\def\sevenpoint{\def\rm{\fam0\sevenrm}%
        \textfont0=\sevenrm \scriptfont0=\fiverm \scriptscriptfont0=\fiverm
        \textfont1=\seveni \scriptfont1=\fivei \scriptscriptfont1=\fivei
        \textfont2=\sevensy \scriptfont2=\fivesy \scriptscriptfont2=\fivesy
        \textfont3=\tenex \scriptfont3=\tenex \scriptscriptfont3=\tenex
        \def\it{\fam\itfam\sevenit}%
        \textfont\itfam=\sevenit
        \def\sl{\fam\slfam\sevensl}%
        \textfont\slfam=\sevensl
        \def\bf{\fam\bffam\sevenbf}%
        \textfont\bffam=\sevenbf \scriptfont\bffam=\fivebf
                \scriptscriptfont\bffam=\fivebf
        \def\tt{\fam\ttfam\seventt}%
        \textfont\ttfam=\seventt
        \normalbaselineskip=8pt%
        \let\sc=\fiverm  
        \setbox\strutbox=\hbox{\vrule height6pt depth2pt width0pt}%
        \normalbaselines\rm}

\input pictex
\font\smallsymbol = cmmi8
\newdimen\xfiglen \newdimen\yfiglen

\begin{document}
\title{
Regularising the Cauchy problem for Laplace's equation by fractional operators}
\author{Barbara Kaltenbacher\footnote{
Department of Mathematics,
Alpen-Adria-Universit\"at Klagenfurt.
barbara.kaltenbacher@aau.at.}
\and
William Rundell\footnote{
Department of Mathematics,
Texas A\&M University,
Texas 77843. 
rundell@tamu.edu}
}
\date{\vskip-3ex}
 \maketitle  
\leftline{\small \qquad\qquad
\textbf{{\textsc ams}} {\bf classification:} 35J25, 35R11, 35R30, 35R35, 65J20.}

\begin{abstract}
In this paper we revisit the classical Cauchy problem for Laplace's equation
 as well as two further related problems in the light of regularisation 
of this highly ill-conditioned problem
by replacing integer derivatives with fractional ones. 
We do so in the spirit of quasi reversibility, replacing a classically severely ill-posed PDE problem by a nearby well-posed or only mildly ill-posed one.
In order to be able to make use of the known stabilising effect of one-dimensional
fractional derivatives of Abel type we work in a particular
rectangular (in higher space dimensions cylindrical) geometry. 
We start with the plain Cauchy problem of reconstructing the values of a harmonic function inside this domain from its Dirichlet and Neumann trace on part of the boundary (the cylinder base) and explore three options for doing this with fractional operators.
The two other related problems are the recovery of a free boundary and then
this together with simultaneous recovery of the impedance function in the
boundary condition.
Our main technique here will be Newton's method.
The paper contains numerical reconstructions and convergence results for the devised methods. 

\end{abstract}

\setcounter{section}{-1}
\section{Introduction}


As its name suggests, the Cauchy Problem for Laplace's equation
has a long history.
By the early-middle of the nineteenth century it was known that
prescribing the values $u$ on the boundary $\partial\Omega$
of a domain $\Omega$ where $-\triangle u = 0$ held,
 allowed $u$ to be determined uniquely within $\Omega$.
There was a similar statement for ``flux'' or the value
of the normal derivative: the so-called Dirichlet and Neumann problems.
These problems held great significance for an enormous range of applications
evident at that time and provided solutions that depended continuously
on the boundary measurements.
It was also recognised that frequently a case would arise where part
of the boundary is inaccessible and no measurements could be made there.
In compensation one could measure both the value and the flux at the accessible
part: the Cauchy problem.
Since solutions of Laplace's equation can be considered as the real part
of an underlying analytic functions, analytic continuation still allowed
uniqueness of the solution.
However, the continuous dependence on the boundary data was lost;
in fact in an extreme way.

A famous reference to this state of affairs  dates from the beginning
of the twentieth century when Hadamard singled such problems
out as being ``incorrectly set'' and  hence unworthy of mathematical study,
as they had ``no physical significance.''
The backwards heat problem and the Cauchy problem were the prime exhibits,
\cite{Hadamard:1902,Hadamard:1907,Hadamard:1923}.

By the middle of the twentieth century such problems had shown to have
enormous physical significance and could not be ignored from any perspective.
Methods had to be found to overcome the severe ill-conditioning.
During this period the subject was extended to general inverse problems
and included a vast range of situations for which the inverse map is an
unbounded operator.
Examples of still foundational papers from this period including applications,
are \cite{Calderon:1958,Payne:1961,Payne:1966,Payne:1970}.

One of the popular techniques dating from this period is the
{\it method of quasi-reversibility\/}
of Lattes and Lions, \cite{LattesLions:1969}.
In this approach the original partial differential equation was replaced by one
in which the ``incorrect'' data allowed a well-posed recovery of its solution.
This new equation contained a parameter $\epsilon$ that allowed for
stable inversion for $\epsilon>0$ but, in addition (in a sense that had to be
carefully defined), solutions of the regularising equation converged to that
of the original as $\epsilon\to 0$.
It is now recognised that the original initial suggested choices of that time
brought with them new problematic issues either because of 
additional unnatural boundary conditions required or an operator
whose solutions behaved in a strongly different manner from the original
that offset any regularising amelioration that it offered.
Thus, the method came with a basic and significant challenge:
finding a ``closely-related''
partial differential equation, depending on a parameter $\epsilon$,
that could use the data in a well-posed manner for $\epsilon>0$
and also be such that its solutions converged to those of the original equation
as $\epsilon\to 0$.

In other words, the central issue in using quasi-reversibility
is in the choice of the regularising equation.
Here we will follow recent ideas for the backwards heat problem
and replace the usual derivative in the ``difficult'' direction
by a fractional derivative.
In the parabolic case this was a time fractional derivative and one
of the first papers taking this direction was \cite{LiuYamamoto:2010}.
It was later shown in \cite{JinRundell:2015} that the effectiveness of the
method and the choice of the fractional exponent used strongly
depended on both the final time $T$ and the frequency components in
the initial time function $u_0(x)$.
This led to the current authors proposing a multi-level version
with different fractional exponents depending on the frequencies in $u_0$,
which of course had also to appear in the measured final time and
were thus identifable, \cite{fracderiv}.
This ``split frequency'' method will also be used in the current
paper.

The advantage here is that such fractional operators arise naturally.
The diffusion equation in $(x,y)-$ coordinates results from a diffusive
process in which the underlying stochastic process arises form sampling
through a probabilty density function $\psi$.
If $\psi$ function has both finite mean and variance then it can be shown
that the long term limit approaches Brownian motion resulting in
classical derivatives.
This can be viewed as a direct result of the central limit theorem.
Allowing a finite mean but an infinite variance can lead directly to
fractional derivatives, \cite{BBB}.

However, while the use of time fractional derivatives and their behaviour
in resulting partial differential equations is now well-understood,
the same cannot be said to the same degree for the case of space
fractional derivatives.
This statement notwithstanding it can identify the standard derivative
as a limiting situation of one of fractional type.
This makes such an operator a natural candidate for a quasi-reversibility 
operator.

\smallskip
In our case if $\Omega$ is a rectangle with the top side inaccessible
but data can be measured on the other three sides then a basic problem
is to recover the solution $u(x,y)$ by also measuring the flux
$\frac{\partial u}{\partial y}$ at, say, the bottom edge.
In the usual language we have Dirichlet data on the two sides and Cauchy data
on the bottom.
We will consider this problem as Problem 1 in Section~\ref{sec:Cauchy1}.

However there is a further possibility: the top side may be a curve $\ell(x)$
and we also do not know this curve -- and want to do so.
This is a classical example of a free-boundary value problem and
a typical, and well studied, example here is of corrosion to a partly
inaccessible metal plate. 
This is our Problem 2 (cf. Section~\ref{sec:Cauchy2}).

In addition, while the original top side was, say,
a pure conductor or insulator  with either $u=0$ or
$\frac{\partial u}{\partial \nu}=0$ there,
this has now to be re-modelled as an impedance condition where the
impedance parameter is also likely unknown as a function of $x$.
Recovery of both the boundary curve and the impedance coefficient is the topic of  Problem 3 in Section~\ref{sec:Cauchy3}.

\medskip

Related to this are obstacle problems for elliptic problems in a domain $\Omega$
that  seek to recover an interior object $D$ from additional
boundary data.
This comes under this same classification, albeit with a different geometry.
The boundary of $D$ can be purely conductive, or purely insulating,
or satisfy an impedance condition with a perhaps unknown parameter.
The existing literature here is again extensive.
We mention the survey by Isakov, \cite{Isakov:2009} which includes not
only elliptic but also  parabolic and hyperbolic equation-based  problems.
Other significant papers from this time period are
\cite{Alessandrini:1997,KlibanovSantosa:1991,CakoniKress:2012,Rundell2008a,Bacchelli:2009};
a very recent overview on numerical methods for the Cauchy problem with further references can be found in \cite{BotchevKabanikhinShishleninTyrtyshnikov:2023}.

For our purposes we wish to take advantage of the geometry described earlier
where we are able, in some sense, to separate the variables and treat each of
the differential operator's components in a distinct manner.

\medskip
The stucture of the paper is as follows.
Each of the three sections~\ref{sec:Cauchy1}, \ref{sec:Cauchy2}, \ref{sec:Cauchy3} first of all contains a formulation of the problem along with the derivation of a reconstruction method and numerical reconstruction results. In Section~\ref{sec:Cauchy1}, this is a quasi reversibility approach based on fractional derivatives; in Sections~\ref{sec:Cauchy2}, \ref{sec:Cauchy3} dealing with nonlinear problems, these are regularised Newton type methods. Sections~\ref{sec:Cauchy1}, \ref{sec:Cauchy3} also contain convergence results. In particular, in Section~\ref{sec:Cauchy3} we verify a range invariance condition on the forward operator that allows us to prove convergence of a regularised frozen Newton method.

\section{Problem 1}\label{sec:Cauchy1}
\begin{problem}\label{Cauchy1}
Given $f$, $g$ in a region $\Omega$
\begin{equation}\label{eqn:Cauchyproblem}
\begin{aligned}
&-\triangle u=-\triangle_x u -\partial_y^2 u=0  \mbox{ in }\Omega\times(0,\olell)\\
&u(\cdot,y)=0\mbox{ on }\partial\Omega\times(0,\olell)\\
&u(x,0)=f(x)\,, \partial_y u(x,0)=g(x)\, \quad x\in\Omega
\end{aligned}
\end{equation}
find $u(x,y)$ in the whole cylinder $\Omega\times(0,\olell)$.
\end{problem}

This is a classical inverse problem going back to before Hadamard \cite{Hadamard:1923} and there exists a huge amount of literature on it.
For a recent review 
and further references, see, e.g., \cite{BotchevKabanikhinShishleninTyrtyshnikov:2023}.

Expansion of $u(\cdot,y)$, $f$, $g$ with respect to the eigenfunctions $\phi_j$ (with corresponding eigenvalues $\lambda_j$) of $-\triangle_x$ on $\Omega$ with homogeneous Dirichlet boundary conditions yields
\[
u(x,y)=\sum_{j=1}^\infty u_j(y) \phi_j(x)\,, \quad 
f(x)=\sum_{j=1}^\infty f_j \phi_j(x)\,, \quad 
g(x)=\sum_{j=1}^\infty g_j \phi_j(x)\,, \quad 
\]
where for all $j\in\mathbb{N}$
\begin{equation}\label{eqn:uj}
u_j''(y)-\lambda_j u_j(y)=0 \, \quad y\in(0,\olell)\,, \quad
u_j(0)=f_j\,, \quad u_j'(0)=g_j\,.
\end{equation}
Thus 
\begin{equation}\label{eqn:uellFourier}
u(x,y)=\sum_{j=1}^\infty a_j \phi_j(x)\,,
\end{equation}
where 
\begin{equation}\label{eqn:aj}
\begin{aligned}
a_j&=f_j\cosh(\sqrt{\lambda_j}y)+g_j\frac{\sinh(\sqrt{\lambda_j}y)}{\sqrt{\lambda_j}}\\
&=\frac{\sqrt{\lambda_j} f_j+g_j}{2\sqrt{\lambda_j}}\exp(\sqrt{\lambda_j}y)
+\frac{\sqrt{\lambda_j} f_j-g_j}{2\sqrt{\lambda_j}}\exp(-\sqrt{\lambda_j}y).
\end{aligned}
\end{equation}

Here the negative Laplacian $-\triangle_x$ on $\Omega$
with homogeneous Dirichlet boundary conditions can obviously be replaced by an
arbitrary symmetric positive definite operator acting on a Hilbert space, in
particular by an elliptic differential operator with possibly $x$ dependent
coefficients on a $d$-dimensional Lipschitz domain $\Omega$ with homogeneous
Dirichlet, Neumann or impedance boundary conditions.

\subsection{Regularisation by fractional differentiation}
Since the values of $u$ have to be propagated in the $y$ direction,
starting from the data $f$, $g$ at $y=0$, the reason for ill-posedness 
(as is clearly visible in the exponential amplification of noise in this data,
cf. \eqref{eqn:aj}), results from the $y$-derivative in the PDE. 
We thus consider several options of regularising Problem~\ref{Cauchy1} by replacing the second order derivative with respect to $y$ by a fractional one, in the spirit of quasi reversibility
\cite{AmesClark_etal:1998,ClarkOppenheimer:1994,LattesLions:1969,Showalter:1974a,Showalter:1974b,Showalter:1976}. We note in particular \cite[Sections 8.3, 10.1]{BBB} in the context of fractional derivatives.

In order to make use of integer (0th and 1st) order derivative data at $y=0$, we use the Djrbashian-Caputo (rather than the Riemann-Liouville) version of the
Abel fractional derivative.
This has a left- and a right-sided version defined by 
\[
{}_0D_y^{\beta} v = h_{2-\beta}*\partial_y^2 v, \quad 
\overline{{}_yD_{\olell}^{\beta} v}^{\olell} = h_{2-\beta}*\partial_y^2 \overline{v}^{\olell}
, \quad 
h_{2-\beta}(y)=\frac{1}{\Gamma(2-\beta)\, y^{\beta-1}}, \quad 
\overline{v}^{\olell}(y)=v(\olell-y)
\]
for $\beta\in(1,2)$, where $*$ denotes the (Laplace) convolution. Note that the Laplace transform of $h_{2-\beta}$ is given by $\widehat{h}_{2-\beta}(s)=s^{\beta-2}$.
Correspondingly, as solutions to initial value problems for fractional ODEs,
Mittag-Leffler functions, as defined by (see, e.g., \cite[Section 3.4]{BBB})
\begin{equation*}
  E_{\alpha,\beta}(z) = \sum_{k=0}^\infty \frac{z^k}{\Gamma(\alpha k+\beta)}\quad z\in \mathbb{C},
\end{equation*}
for $\alpha>0$, and $\beta\in\mathbb{R}$ will play a role.

While using the spectral decomposition 
\begin{equation}\label{eqn:ualpha}
u^{(\alpha)}(x,y)=\sum_{j=1}^\infty a^{(\alpha)}_j \phi_j(x)\,,
\end{equation}
to approximate \eqref{eqn:uellFourier}, \eqref{eqn:aj} in the analysis, the computational implementation does not need the eigenvalues and eigenfunctions of $-\triangle_x$ but relies on the numerical solution of fractional PDEs, for which efficient methods are available, see, e.g., \cite{Alikhanov:2015,JinLazarovZhou:2013,LanglandsHenry:2005,LinXu:2007}. 

\subsubsection*{Left-sided Djrbashian-Caputo fractional derivative} \label{subsec:leftDC}
Replacing $\partial_y^2$ by ${}_0D_y^{2\alpha}$ with $2\alpha\approx2$ 
amounts to considering, instead of \eqref{eqn:uj}, the fractional ODEs
\[
{}_0D_y^{2\alpha} u_j(y)-\lambda_j u_j(y)=0 \, \quad y\in(0,\olell)\,, \quad
u_j(0)=f_j\,, \quad u_j'(0)=g_j,
\]
whose solution by means of Mittag-Leffler functions (see, e.g., \cite[Theorem 5.4]{BBB})
yields
\begin{equation}\label{eqn:leftDC}
a_j^{(\alpha)}=f_jE_{2\alpha,1}(\lambda_j y^{2\alpha})+g_j y E_{2\alpha,2}(\lambda_j y^{2\alpha}).
\end{equation}
In view of the fact that $E_{2,1}(z) =\cosh(\sqrt{z})$ and $E_{2,2}(z) = \frac{\sinh \sqrt{z}}{\sqrt{z}}$, this is consistent with \eqref{eqn:aj}.

\subsubsection*{Right-sided Djrbashian-Caputo fractional derivative} \label{subsec:rightDC}

Replacing $\partial_y^2$ by ${}_yD_{\olell}^{2\alpha}$ with $2\alpha\approx2$ corresponds to replacing \eqref{eqn:uj} by 
\[
{}_yD_{\olell}^{2\alpha} u_j(y)-\lambda_j u_j(y)=0 \, \quad y\in(0,\olell)\,, \quad
u_j(0)=f_j\,, \quad u_j'(0)=g_j
\]
together with
\[
u_j(\olell)=\bar{a}_j\,, \quad u_j'(\olell)=\bar{b}_j\,.
\]
From the identity 
\[
{}_yD_{\olell}^{2\alpha} u_j(\olell-\eta)=(h_{2-2\alpha}*w_j'')(\eta)
\quad\mbox{ for } w_j(\eta)=u_j(\olell-\eta)\mbox{ and } \widehat{h}_{2-2\alpha}(s)=s^{2\alpha-2}
\]
we obtain the initial value problem
\[
(h_{2-2\alpha}*w_j'')(\eta)-\lambda_j w_j(\eta)=0\, \quad \eta\in(0,\olell)\,, \quad
w_j(0)=\bar{a}_j\,, \quad w_j'(0)=-\bar{b}_j\,.
\]
Taking Laplace transforms yields
\[
s^{2\alpha} \widehat{w}_j(s)-s^{2\alpha-1} \bar{a}_j+s^{2\alpha-2} \bar{b}_j-\lambda_j \widehat{w}_j(s)=0
\]
i.e.,
\begin{equation}\label{eqn:what}
\widehat{w}_j(s)=\frac{s^{2\alpha-1}}{s^{2\alpha}-\lambda_j} \bar{a}_j - \frac{s^{2\alpha-2}}{s^{2\alpha}-\lambda_j} \bar{b}_j,
\end{equation}
and for the derivative
\begin{equation}\label{eqn:wprimehat}
\widehat{w_j'}(s)= s\widehat{w}_j(s)-\bar{a}_j=
\frac{\lambda_j}{s^{2\alpha}-\lambda_j} \bar{a}_j - \frac{s^{2\alpha-1}}{s^{2\alpha}-\lambda_j} \bar{b}_j.
\end{equation}
From 
\cite[Lemma 4.12]{BBB}
we obtain
\[
\mathcal{L}(\eta^{k-1} E_{2\alpha,k}(\lambda \eta^{2\alpha}))(s)=\frac{s^{2\alpha-k}}{s^{2\alpha}-\lambda_j}\,, \ k\in\{1,2\}
\,, \quad 
\mathcal{L}(-\eta^{2\alpha-1} E_{2\alpha,2\alpha}(\lambda \eta^{2\alpha}))(s)=\frac{1}{s^{2\alpha}-\lambda_j}\,.
\]
Inserting this into \eqref{eqn:what}, \eqref{eqn:wprimehat} and evaluating at $\eta=\olell$ we obtain
\[
\begin{aligned}
f_j=& E_{2\alpha,1}(\lambda_j \olell^{2\alpha}) \bar{a}_j 
    - \olell E_{2\alpha,2}(\lambda_j \olell^{2\alpha}) \bar{b}_j\\
-g_j=&-\lambda_j \olell^{2\alpha-1} E_{2\alpha,2\alpha}(\lambda_j \olell^{2\alpha}) \bar{a}_j 
    - E_{2\alpha,1}(\lambda_j \olell^{2\alpha}) \bar{b}_j\,.
\end{aligned}
\]
Resolving for $\bar{a}_j$ and replacing $\olell$ by $y$ we get
\begin{equation}\label{eqn:rightDC}
a_j^{(\alpha)}=\frac{f_jE_{2\alpha,1}(\lambda_j y^{2\alpha})+g_j y E_{2\alpha,2}(\lambda_j y^{2\alpha})}{
\bigl(E_{2\alpha,1}(\lambda_j y^{2\alpha})\bigr)^2-\lambda_j y^{2\alpha} E_{2\alpha,2\alpha}(\lambda_j y^{2\alpha}) E_{2\alpha,2}(\lambda_j y^{2\alpha})}.
\end{equation}

\subsubsection*{Factorisation of the Laplacian}\label{subsec:facLap}
An analysis of the two one-sided fractional approximations of $\partial_y^2$ does not seem to be possible since it would require a stability estimate for Mittag-Leffler functions with positive argument and index close to two.
While convergence from below of the fractional to the integer derivative holds at any integer (thus also second) order, a stability estimate is not available. Therefore we look for a possibility to reduce the problem to one with first order $y$ derivatives (and treat the inverse problem similarly to a backwards heat problem to take advantage of recent work in this direction \cite{fracderiv,BBB}). One way to do so is to factorise the negative Laplacian so that the Cauchy problem becomes:\hfill\break
Given $f$, $g$ in
\[
\begin{aligned}
&-\triangle u=-\triangle_x u-\partial_y^2 u=(\partial_y -\sqrt{-\triangle_x})(-\partial_y -\sqrt{-\triangle_x})u=0  \mbox{ in }\Omega\times(0,\olell)\\
&u(\cdot,y)=0\mbox{ on }\partial\Omega\times(0,\olell)\\
&u(x,0)=f(x)\,, \partial_y u(x,0)=g(x)\, \quad x\in\Omega
\end{aligned}
\]
find $u(x,y)$.\hfill\break
More precisely, with
$u_\pm= \frac12(u\pm\sqrt{-\triangle_x}^{\,-1}\partial_yu)$
we get the representation
\[
u=u_++u_-
\] 
where $u_+$, $u_-$ can be obtained as solutions to the subproblems
\begin{equation}\label{eqn:uplus}
\begin{aligned}
&\partial_y u_+-\sqrt{-\triangle_x} u_+=0  \mbox{ in }\Omega\times(0,\olell)\\
&u_+(\cdot,y)=0\mbox{ on }\partial\Omega\times(0,\olell) 
\\
&u_+(x,0)=\tfrac12(f(x)+\sqrt{-\triangle_x}^{\,-1}g(x))=:u_{+0}(x)\, \quad x\in\Omega
\end{aligned}
\end{equation}
and 
\begin{equation}\label{eqn:uminus}
\begin{aligned}
&\partial_y u_-+\sqrt{-\triangle_x} u_-=0  \mbox{ in }\Omega\times(0,\olell)\\
&u_-(\cdot,y)=0\mbox{ on }\partial\Omega\times(0,\olell)
\\
&u_-(x,0)=\tfrac12(f(x)-\sqrt{-\triangle_x}^{\,-1}g(x))\, \quad x\in\Omega\,.
\end{aligned}
\end{equation}
In fact, it is readily checked that if $u_\pm$ solve \eqref{eqn:uplus}, \eqref{eqn:uminus} then $u=u_++u_-$ solves \eqref{eqn:Cauchyproblem}.
The numerical solution of the initial value problem \eqref{eqn:uminus} and of the final value problem for the PDE in \eqref{eqn:uplus} can be stably and efficiently carried out combining an implicit time stepping scheme with methods recently developed for the solution of PDEs with fractional powers of the Laplacian.
See, e.g., \cite{Bonitoetal:2018,BonitoLeiPasciak:2019,HarizanovLazarovMargenov:2020}. 

Since $\sqrt{-\triangle_x}$ is positive definite, the second equation \eqref{eqn:uminus} is well-posed, so there is no need to regularise.
The first one \eqref{eqn:uplus} after the change of variables $t=\olell-y$, $w_+(\olell-t)=u_+(y)$
becomes a backwards heat equation (but with $\sqrt{-\triangle}$ in place of $-\triangle$)
\begin{equation}\label{eqn:wplus}
\begin{aligned}
&\partial_t w_++\sqrt{-\triangle_x} w_+=0  \mbox{ in }\Omega\times(0,\olell)\\
&w_+(\cdot,t)=0\mbox{ on }\partial\Omega\times(0,\olell)
\\
&w_+(x,\olell)=\tfrac12(f(x)+\sqrt{-\triangle_x}^{\,-1}g(x))\, \quad x\in\Omega
\end{aligned}
\end{equation}
Regularizing \eqref{eqn:wplus} by using in place of $\partial_t$ a fractional
``time'' derivative ${}_0D_t^{\alpha}$  with ${\alpha}\approx 1$, $\alpha<1$ (while leaving \eqref{eqn:uminus} unregularised) amounts to setting
\begin{equation}\label{eqn:facLap}
a_j^{\alpha}=\frac{\sqrt{\lambda_j} f_j+g_j}{2\sqrt{\lambda_j}}\frac{1}{E_{\alpha,1}(-\sqrt{\lambda_j}y^{\alpha})}
+\frac{\sqrt{\lambda_j} f_j-g_j}{2\sqrt{\lambda_j}}\exp(-\sqrt{\lambda_j}y)\,.
\end{equation}

This approach can be refined by splitting the frequency range $(\lambda_j)_{j\in\mathbb{N}}$ into subsets 
$(\{\lambda_{K_i+1},\ldots,\lambda_{K_{i+1}}\})_{i\in\mathbb{N}}$ and choosing the breakpoint $K_i$ as well as the fractional order $\alpha_i$ for each of these subsets according to a discrepancy principle. For details, see \cite{fracderiv}.

\medskip

\subsection{Reconstructions}\label{sec:reconCauchy1}
In this section we compare reconstructions with the three options \eqref{eqn:leftDC}, \eqref{eqn:rightDC}, \eqref{eqn:facLap}.
The latter was refined by the split frequency approach from \cite{fracderiv} using the discrepancy principle for determining the breakpoints and differentiation orders.
While this method is backed up by convergence theory,
the same does not hold true for the options \eqref{eqn:leftDC} and
\eqref{eqn:rightDC}.
Indeed, not even stability can be expected to hold from the known behaviour of
the Mittag Leffler functions with positive argument, in particular for \eqref{eqn:leftDC}.
This becomes visible in the reconstructions in Figure~\ref{fig:Cauchy1}.  
The differentiation orders for \eqref{eqn:leftDC}, \eqref{eqn:rightDC} were taken as the smallest (thus most stable) ones obtained in \eqref{eqn:facLap}.

\xfiglen=1.5 true in
\yfiglen=1.2true in
\newbox\figurelegendCauchyone
\newbox\figboxten
\newbox\figboxtt
\newbox\figboxss

\setbox\figboxten=\vbox{\hsize=\xfiglen  
\beginpicture
  \setcoordinatesystem units <\xfiglen,\yfiglen>  point at 0.0 0.0
  \setplotarea x from 0 to 1, y from 0.0 to 1.5
\relax
\scriptsize
  \axis bottom ticks short numbered from 0 to 1 by 0.2 /
  \axis left ticks short numbered from 0 to 1.5 by 0.5 /
\setquadratic
\setsolid
\linethickness=0.6pt
\Black{\relax 
\plot
         0         0
    0.0200    0.1629
    0.0400    0.3213
    0.0600    0.4712
    0.0800    0.6092
    0.1000    0.7331
    0.1200    0.8419
    0.1400    0.9357
    0.1600    1.0156
    0.1800    1.0836
    0.2000    1.1420
    0.2200    1.1932
    0.2400    1.2391
    0.2600    1.2815
    0.2800    1.3210
    0.3000    1.3576
    0.3200    1.3906
    0.3400    1.4190
    0.3600    1.4414
    0.3800    1.4566
    0.4000    1.4638
    0.4200    1.4629
    0.4400    1.4545
    0.4600    1.4402
    0.4800    1.4222
    0.5000    1.4030
    0.5200    1.3853
    0.5400    1.3717
    0.5600    1.3639
    0.5800    1.3629
    0.6000    1.3686
    0.6200    1.3795
    0.6400    1.3932
    0.6600    1.4065
    0.6800    1.4155
    0.7000    1.4164
    0.7200    1.4055
    0.7400    1.3798
    0.7600    1.3375
    0.7800    1.2777
    0.8000    1.2009
    0.8200    1.1085
    0.8400    1.0031
    0.8600    0.8875
    0.8800    0.7648
    0.9000    0.6379
    0.9200    0.5093
    0.9400    0.3806
    0.9600    0.2528
    0.9800    0.1260
    1.0000    0.0000
/\relax}\relax
\linethickness=0.6pt
\setdashpattern <5pt, 2pt, 3pt, 2pt>
\OliveGreen{\relax
\plot   
         0         0
    0.0200    0.1629
    0.0400    0.3213
    0.0600    0.4711
    0.0800    0.6091
    0.1000    0.7331
    0.1200    0.8419
    0.1400    0.9358
    0.1600    1.0157
    0.1800    1.0838
    0.2000    1.1422
    0.2200    1.1934
    0.2400    1.2393
    0.2600    1.2816
    0.2800    1.3210
    0.3000    1.3575
    0.3200    1.3905
    0.3400    1.4188
    0.3600    1.4411
    0.3800    1.4562
    0.4000    1.4633
    0.4200    1.4624
    0.4400    1.4541
    0.4600    1.4399
    0.4800    1.4220
    0.5000    1.4029
    0.5200    1.3854
    0.5400    1.3720
    0.5600    1.3644
    0.5800    1.3636
    0.6000    1.3693
    0.6200    1.3803
    0.6400    1.3941
    0.6600    1.4073
    0.6800    1.4163
    0.7000    1.4171
    0.7200    1.4060
    0.7400    1.3802
    0.7600    1.3378
    0.7800    1.2778
    0.8000    1.2008
    0.8200    1.1084
    0.8400    1.0029
    0.8600    0.8872
    0.8800    0.7645
    0.9000    0.6376
    0.9200    0.5090
    0.9400    0.3803
    0.9600    0.2526
    0.9800    0.1260
    1.0000    0.0000
 /\relax}\relax
\setdashes <4pt>
\Red{\relax
\plot 
         0         0
    0.0200    0.1718
    0.0400    0.3386
    0.0600    0.4958
    0.0800    0.6399
    0.1000    0.7684
    0.1200    0.8802
    0.1400    0.9754
    0.1600    1.0554
    0.1800    1.1225
    0.2000    1.1792
    0.2200    1.2282
    0.2400    1.2721
    0.2600    1.3124
    0.2800    1.3501
    0.3000    1.3851
    0.3200    1.4168
    0.3400    1.4439
    0.3600    1.4647
    0.3800    1.4780
    0.4000    1.4830
    0.4200    1.4793
    0.4400    1.4679
    0.4600    1.4503
    0.4800    1.4291
    0.5000    1.4072
    0.5200    1.3876
    0.5400    1.3730
    0.5600    1.3657
    0.5800    1.3666
    0.6000    1.3755
    0.6200    1.3910
    0.6400    1.4103
    0.6600    1.4297
    0.6800    1.4449
    0.7000    1.4515
    0.7200    1.4454
    0.7400    1.4233
    0.7600    1.3830
    0.7800    1.3236
    0.8000    1.2455
    0.8200    1.1505
    0.8400    1.0413
    0.8600    0.9210
    0.8800    0.7932
    0.9000    0.6610
    0.9200    0.5272
    0.9400    0.3936
    0.9600    0.2612
    0.9800    0.1302
    1.0000    0.0000
 /\relax}\relax
\setdots <2pt>
\Blue{\relax
\plot 
         0         0
    0.0200    0.1940
    0.0400    0.3807
    0.0600    0.5538
    0.0800    0.7080
    0.1000    0.8404
    0.1200    0.9498
    0.1400    1.0374
    0.1600    1.1059
    0.1800    1.1595
    0.2000    1.2029
    0.2200    1.2406
    0.2400    1.2765
    0.2600    1.3133
    0.2800    1.3520
    0.3000    1.3920
    0.3200    1.4314
    0.3400    1.4673
    0.3600    1.4963
    0.3800    1.5153
    0.4000    1.5222
    0.4200    1.5159
    0.4400    1.4973
    0.4600    1.4685
    0.4800    1.4333
    0.5000    1.3965
    0.5200    1.3632
    0.5400    1.3380
    0.5600    1.3248
    0.5800    1.3257
    0.6000    1.3409
    0.6200    1.3685
    0.6400    1.4045
    0.6600    1.4434
    0.6800    1.4788
    0.7000    1.5040
    0.7200    1.5129
    0.7400    1.5005
    0.7600    1.4637
    0.7800    1.4015
    0.8000    1.3149
    0.8200    1.2069
    0.8400    1.0820
    0.8600    0.9456
    0.8800    0.8030
    0.9000    0.6592
    0.9200    0.5178
    0.9400    0.3813
    0.9600    0.2503
    0.9800    0.1238
    1.0000    0.0000
 /\relax}\relax
\endpicture
}

\setbox\figboxtt=\vbox{\hsize=\xfiglen  
\beginpicture
  \setcoordinatesystem units <\xfiglen,\yfiglen>  point at 0.0 0.0
  \setplotarea x from 0 to 1, y from 0.0 to 1.5
\relax
\scriptsize
  \axis bottom ticks short numbered from 0 to 1 by 0.2 /
  \axis left ticks short numbered from 0 to 1.5 by 0.5 /
\setquadratic
\setsolid
\linethickness=0.6pt
\Black{\relax 
\plot
         0         0
    0.0200    0.1272
    0.0400    0.2521
    0.0600    0.3726
    0.0800    0.4869
    0.1000    0.5936
    0.1200    0.6919
    0.1400    0.7812
    0.1600    0.8617
    0.1800    0.9336
    0.2000    0.9977
    0.2200    1.0547
    0.2400    1.1053
    0.2600    1.1502
    0.2800    1.1898
    0.3000    1.2243
    0.3200    1.2540
    0.3400    1.2787
    0.3600    1.2983
    0.3800    1.3129
    0.4000    1.3225
    0.4200    1.3275
    0.4400    1.3284
    0.4600    1.3259
    0.4800    1.3211
    0.5000    1.3149
    0.5200    1.3083
    0.5400    1.3021
    0.5600    1.2969
    0.5800    1.2928
    0.6000    1.2895
    0.6200    1.2861
    0.6400    1.2816
    0.6600    1.2743
    0.6800    1.2626
    0.7000    1.2448
    0.7200    1.2191
    0.7400    1.1843
    0.7600    1.1394
    0.7800    1.0840
    0.8000    1.0182
    0.8200    0.9423
    0.8400    0.8573
    0.8600    0.7645
    0.8800    0.6651
    0.9000    0.5606
    0.9200    0.4522
    0.9400    0.3412
    0.9600    0.2283
    0.9800    0.1144
    1.0000    0.0000
/\relax}\relax
\linethickness=0.6pt
\setdashpattern <5pt, 2pt, 3pt, 2pt>
\OliveGreen{\relax
\plot   
         0         0
    0.0200    0.1272
    0.0400    0.2521
    0.0600    0.3726
    0.0800    0.4869
    0.1000    0.5936
    0.1200    0.6919
    0.1400    0.7812
    0.1600    0.8617
    0.1800    0.9337
    0.2000    0.9978
    0.2200    1.0548
    0.2400    1.1054
    0.2600    1.1502
    0.2800    1.1898
    0.3000    1.2243
    0.3200    1.2539
    0.3400    1.2786
    0.3600    1.2982
    0.3800    1.3128
    0.4000    1.3224
    0.4200    1.3274
    0.4400    1.3283
    0.4600    1.3258
    0.4800    1.3210
    0.5000    1.3148
    0.5200    1.3083
    0.5400    1.3022
    0.5600    1.2971
    0.5800    1.2930
    0.6000    1.2897
    0.6200    1.2864
    0.6400    1.2818
    0.6600    1.2746
    0.6800    1.2628
    0.7000    1.2450
    0.7200    1.2193
    0.7400    1.1844
    0.7600    1.1395
    0.7800    1.0841
    0.8000    1.0182
    0.8200    0.9423
    0.8400    0.8573
    0.8600    0.7644
    0.8800    0.6651
    0.9000    0.5605
    0.9200    0.4522
    0.9400    0.3411
    0.9600    0.2283
    0.9800    0.1144
    1.0000    0.0000
 /\relax}\relax
\setdashes <4pt>
\Red{\relax
\plot 
         0         0
    0.0200    0.1291
    0.0400    0.2558
    0.0600    0.3779
    0.0800    0.4934
    0.1000    0.6010
    0.1200    0.6996
    0.1400    0.7890
    0.1600    0.8692
    0.1800    0.9407
    0.2000    1.0043
    0.2200    1.0607
    0.2400    1.1107
    0.2600    1.1552
    0.2800    1.1945
    0.3000    1.2290
    0.3200    1.2586
    0.3400    1.2833
    0.3600    1.3029
    0.3800    1.3173
    0.4000    1.3266
    0.4200    1.3311
    0.4400    1.3313
    0.4600    1.3280
    0.4800    1.3224
    0.5000    1.3154
    0.5200    1.3082
    0.5400    1.3016
    0.5600    1.2964
    0.5800    1.2926
    0.6000    1.2899
    0.6200    1.2876
    0.6400    1.2843
    0.6600    1.2784
    0.6800    1.2682
    0.7000    1.2516
    0.7200    1.2271
    0.7400    1.1931
    0.7600    1.1486
    0.7800    1.0932
    0.8000    1.0269
    0.8200    0.9503
    0.8400    0.8644
    0.8600    0.7705
    0.8800    0.6699
    0.9000    0.5643
    0.9200    0.4549
    0.9400    0.3430
    0.9600    0.2294
    0.9800    0.1149
    1.0000    0.0000
 /\relax}\relax
\setdots <2pt>
\Blue{\relax
\plot 
         0         0
    0.0200    0.1302
    0.0400    0.2580
    0.0600    0.3809
    0.0800    0.4971
    0.1000    0.6051
    0.1200    0.7040
    0.1400    0.7933
    0.1600    0.8733
    0.1800    0.9444
    0.2000    1.0075
    0.2200    1.0635
    0.2400    1.1133
    0.2600    1.1575
    0.2800    1.1967
    0.3000    1.2312
    0.3200    1.2608
    0.3400    1.2856
    0.3600    1.3052
    0.3800    1.3196
    0.4000    1.3288
    0.4200    1.3330
    0.4400    1.3329
    0.4600    1.3292
    0.4800    1.3230
    0.5000    1.3155
    0.5200    1.3078
    0.5400    1.3010
    0.5600    1.2956
    0.5800    1.2920
    0.6000    1.2897
    0.6200    1.2879
    0.6400    1.2854
    0.6600    1.2804
    0.6800    1.2711
    0.7000    1.2554
    0.7200    1.2315
    0.7400    1.1979
    0.7600    1.1537
    0.7800    1.0983
    0.8000    1.0317
    0.8200    0.9546
    0.8400    0.8681
    0.8600    0.7735
    0.8800    0.6722
    0.9000    0.5660
    0.9200    0.4560
    0.9400    0.3437
    0.9600    0.2298
    0.9800    0.1151
    1.0000    0.0000
 /\relax}\relax
\endpicture
}

\setbox\figboxss=\vbox{\hsize=\xfiglen  
\beginpicture
  \setcoordinatesystem units <\xfiglen,\yfiglen>  point at 0.0 0.0
  \setplotarea x from 0 to 1, y from 0.0 to 1.5
\relax
\scriptsize
  \axis bottom ticks short numbered from 0 to 1 by 0.2 /
  \axis left ticks short numbered from 0 to 1.5 by 0.5 /
\setquadratic
\setsolid
\linethickness=0.6pt
\Black{\relax 
\plot
         0         0
    0.0200    0.1413
    0.0400    0.2795
    0.0600    0.4118
    0.0800    0.5360
    0.1000    0.6503
    0.1200    0.7537
    0.1400    0.8459
    0.1600    0.9274
    0.1800    0.9991
    0.2000    1.0622
    0.2200    1.1179
    0.2400    1.1675
    0.2600    1.2119
    0.2800    1.2516
    0.3000    1.2870
    0.3200    1.3178
    0.3400    1.3437
    0.3600    1.3642
    0.3800    1.3788
    0.4000    1.3874
    0.4200    1.3901
    0.4400    1.3877
    0.4600    1.3811
    0.4800    1.3717
    0.5000    1.3611
    0.5200    1.3508
    0.5400    1.3422
    0.5600    1.3362
    0.5800    1.3334
    0.6000    1.3332
    0.6200    1.3349
    0.6400    1.3367
    0.6600    1.3366
    0.6800    1.3320
    0.7000    1.3205
    0.7200    1.2997
    0.7400    1.2678
    0.7600    1.2234
    0.7800    1.1659
    0.8000    1.0956
    0.8200    1.0133
    0.8400    0.9203
    0.8600    0.8185
    0.8800    0.7098
    0.9000    0.5961
    0.9200    0.4792
    0.9400    0.3604
    0.9600    0.2405
    0.9800    0.1203
    1.0000    0.0000
/\relax}\relax
%
\setdashpattern <5pt, 2pt, 3pt, 2pt>
\OliveGreen{\relax
\plot   
         0         0
    0.0200    0.1413
    0.0400    0.2795
    0.0600    0.4118
    0.0800    0.5360
    0.1000    0.6503
    0.1200    0.7537
    0.1400    0.8460
    0.1600    0.9275
    0.1800    0.9992
    0.2000    1.0623
    0.2200    1.1180
    0.2400    1.1676
    0.2600    1.2119
    0.2800    1.2517
    0.3000    1.2870
    0.3200    1.3178
    0.3400    1.3436
    0.3600    1.3640
    0.3800    1.3785
    0.4000    1.3871
    0.4200    1.3899
    0.4400    1.3875
    0.4600    1.3809
    0.4800    1.3716
    0.5000    1.3611
    0.5200    1.3508
    0.5400    1.3423
    0.5600    1.3365
    0.5800    1.3337
    0.6000    1.3337
    0.6200    1.3354
    0.6400    1.3373
    0.6600    1.3371
    0.6800    1.3325
    0.7000    1.3209
    0.7200    1.3000
    0.7400    1.2680
    0.7600    1.2235
    0.7800    1.1660
    0.8000    1.0956
    0.8200    1.0132
    0.8400    0.9202
    0.8600    0.8184
    0.8800    0.7096
    0.9000    0.5960
    0.9200    0.4791
    0.9400    0.3602
    0.9600    0.2404
    0.9800    0.1203
    1.0000    0.0000
 /\relax}\relax
\setdashes <4pt>
\Red{\relax
\plot 
         0         0
    0.0200    0.1471
    0.0400    0.2908
    0.0600    0.4278
    0.0800    0.5557
    0.1000    0.6725
    0.1200    0.7773
    0.1400    0.8699
    0.1600    0.9508
    0.1800    1.0212
    0.2000    1.0826
    0.2200    1.1366
    0.2400    1.1847
    0.2600    1.2278
    0.2800    1.2668
    0.3000    1.3017
    0.3200    1.3322
    0.3400    1.3579
    0.3600    1.3781
    0.3800    1.3921
    0.4000    1.3996
    0.4200    1.4009
    0.4400    1.3965
    0.4600    1.3876
    0.4800    1.3758
    0.5000    1.3629
    0.5200    1.3507
    0.5400    1.3411
    0.5600    1.3350
    0.5800    1.3331
    0.6000    1.3350
    0.6200    1.3397
    0.6400    1.3453
    0.6600    1.3494
    0.6800    1.3491
    0.7000    1.3416
    0.7200    1.3242
    0.7400    1.2946
    0.7600    1.2514
    0.7800    1.1940
    0.8000    1.1226
    0.8200    1.0381
    0.8400    0.9422
    0.8600    0.8371
    0.8800    0.7249
    0.9000    0.6079
    0.9200    0.4879
    0.9400    0.3663
    0.9600    0.2442
    0.9800    0.1221
    1.0000    0.0000
 /\relax}\relax
\setdots <2pt>
\Blue{\relax
\plot 
         0         0
    0.0200    0.1537
    0.0400    0.3034
    0.0600    0.4454
    0.0800    0.5767
    0.1000    0.6953
    0.1200    0.8001
    0.1400    0.8913
    0.1600    0.9697
    0.1800    1.0370
    0.2000    1.0952
    0.2200    1.1463
    0.2400    1.1922
    0.2600    1.2342
    0.2800    1.2729
    0.3000    1.3086
    0.3200    1.3405
    0.3400    1.3678
    0.3600    1.3894
    0.3800    1.4042
    0.4000    1.4117
    0.4200    1.4119
    0.4400    1.4052
    0.4600    1.3932
    0.4800    1.3777
    0.5000    1.3610
    0.5200    1.3456
    0.5400    1.3334
    0.5600    1.3262
    0.5800    1.3246
    0.6000    1.3286
    0.6200    1.3369
    0.6400    1.3473
    0.6600    1.3569
    0.6800    1.3622
    0.7000    1.3599
    0.7200    1.3467
    0.7400    1.3200
    0.7600    1.2780
    0.7800    1.2201
    0.8000    1.1466
    0.8200    1.0588
    0.8400    0.9588
    0.8600    0.8492
    0.8800    0.7328
    0.9000    0.6122
    0.9200    0.4895
    0.9400    0.3663
    0.9600    0.2436
    0.9800    0.1215
    1.0000    0.0000
 /\relax}\relax
\endpicture
}

\newbox\figurelegendone
\setbox\figurelegendone=\hbox{
\beginpicture
  \setcoordinatesystem units <0.06\xfiglen,0.05\yfiglen> 
  \setplotarea x from 0 to 3, y from 1 to 4
\footnotesize
\linethickness=0.5pt
\scriptsize
   \setsolid
   \Black{\relax \putrule from 0 4 to 1 4 \relax}\relax
   \put {$y$ actual}  [l] at 1.2 4
   \setdashes <4pt>
   \Red{\relax \putrule from 0 3 to 1 3 \relax}\relax
   \put {right der} [l] at 1.2 3
   setdots <2pt>
   \Blue{\relax \putrule from 0 2 to 1 2 \relax}\relax
   \put {left der } [l] at 1.2 2
   \setdashpattern <5pt, 2pt, 3pt, 2pt>
   \OliveGreen{\relax \putrule from 0 1 to 1 1 \relax}\relax
   \put {fac Lapl} [l] at 1.2 1
\endpicture
}

\begin{figure}[ht]
\includegraphics[width=\textwidth]{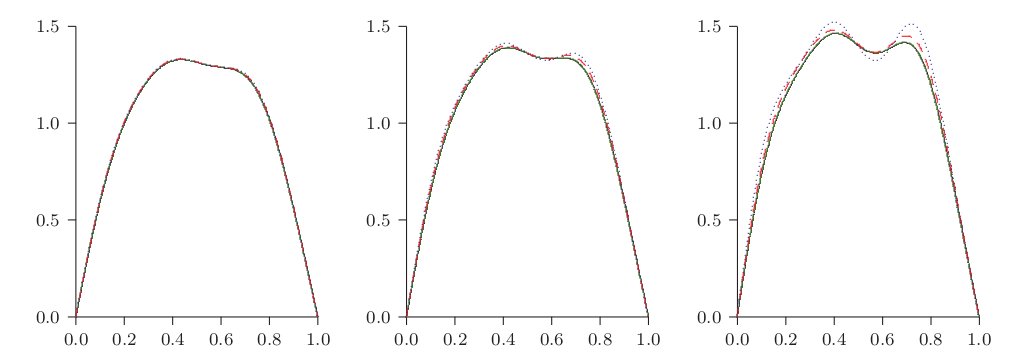}
\caption{\small Reconstructions from formulas \eqref{eqn:leftDC} (blue dotted),
\eqref{eqn:rightDC} (red, dashed) and \eqref{eqn:facLap} (green, irregularly dashed) as compared to the actual value (black, solid) from data with 1 per cent noise at different distances from the Cauchy data boundary.
Left: $y=\frac{1}{3}$,
Middle: $y=\frac{2}{3}$,
Right: $y=1$.
\label{fig:Cauchy1}}
\end{figure}

The relative $L^2$ errors over the whole domain were 
$0.0275$ for \eqref{eqn:leftDC},
$0.0135$ for \eqref{eqn:rightDC},
$1.8597\cdot 10^{-4}$ for \eqref{eqn:facLap}.

The split frequency factorised Laplace approach also worked well with much higher noise levels, as can be seen in Sections~\ref{sec:reconCauchy2} and \ref{sec:reconCauchy3}, where it is intrinsically used as part of the reconstruction algorithm.
On the other hand, \eqref{eqn:leftDC}, \eqref{eqn:rightDC} failed to provide reasonable reconstructions at higher noise levels, which is to be expected from a theoretical point of view.

\subsection{Convergence 
of the scheme \eqref{eqn:facLap}}\label{sec:convCauchy1}
As usual in regularisation methods, the error can be decomposed as 
\begin{equation}\label{eqn:totalerror}
\|u-u^{\alpha,\delta}\|_X\leq \|u-u^{(\alpha)}\|_X+\|u^{(\alpha)}-u^{\alpha,\delta}\|_X
\end{equation}
where $u^{\alpha,\delta}$ is the actual reconstruction from noisy data $f^\delta$, $g^\delta$ leading to 
\begin{equation}\label{eqn:delta}
u_{+0}^\delta = \tfrac12(f^\delta(x)+\sqrt{-\triangle_x}^{\,-1}g^\delta(x))
\text{ with }\|u_{+0}^\delta-u_{+0}\|_Y\leq\delta .
\end{equation}
Using $X=L^2(0,\olell;H^\sigma(\Omega))$ as the space for the sought-after function $u$, we can write the approximation error as 
\begin{equation}\label{eqn:approxerr}
\begin{aligned}
&\|u-u^{(\alpha)}\|_{L^2(0,\olell;H^\sigma(\Omega))} 
=\Bigl(\int_0^{\olell}\sum_{j=1}^\infty \lambda_j^\sigma\, |u_j(y)-u^{(\alpha)}_j(y)|^2\, dy \Bigr)^{1/2}\\
&=\Bigl(  \int_0^{\olell}\sum_{j=1}^\infty |u_{+,j}(y)|^2\, \lambda_j^\sigma \,
\frac{1}{E_{\alpha,1}(-\sqrt{\lambda_j} y^\alpha)^2} 
\left|E_{\alpha,1}(-\sqrt{\lambda_j} y^\alpha)-\exp(-\sqrt{\lambda_j} y)\right|^2 dy   \Bigr)^{1/2}
\end{aligned}
\end{equation}
where $u^\dagger_{+,j}(y)=u_{+0,j} \exp(\sqrt{\lambda_j} y)$,
and the propagated noise term as
\begin{equation}\label{eqn:propnoise}
\begin{aligned}
\|u^{(\alpha)}-u^{\alpha,\delta}\|_{L^2(0,\olell;H^\sigma(\Omega))} 
&=\Bigl(\int_0^{\olell}\sum_{j=1}^\infty \lambda_j^\sigma\, |u^{(\alpha)}_j(y)-u^{\alpha,\delta}_j(y)|^2\, dy \Bigr)^{1/2}\\
&=\Bigl(   \int_0^{\olell}\sum_{j=1}^\infty |u_{+0,j}-u_{+0,j}^\delta|^2\, \lambda_j^\sigma\, \frac{1}{E_{\alpha,1}(-\sqrt{\lambda_j} y^\alpha)^2}  dy   \Bigr)^{1/2}.
\end{aligned}
\end{equation}

In view of \eqref{eqn:wplus}, convergence follows similarly to corresponding results on regularisation by backwards subdiffusion \cite[Section 10.1.3]{BBB} using two fundamental lemmata 
\begin{itemize}
\item
on stability, estimating $\frac{1}{E_{\alpha,1}(-\lambda y^\alpha)}$ (Lemma~\ref{lem:mlf-stability_est}) and 
\item
on convergence, estimating $\left|E_{\alpha,1}(-\sqrt{\lambda_j} y^\alpha)-\exp(-\sqrt{\lambda_j} y)\right|$ (Lemma~\ref{lem:rateE})
\end{itemize}
that we here re-prove to track dependence of constants on the final ``time'' $\TT$. This is important in view of the fact that as opposed to \cite[Section 10.1.3]{BBB}, we consider a range of ``final time values'', that is, in our context, $y$ values.
In the following $\llam$ serves as a placeholder for $\sqrt{\lambda_j}$.

\begin{lemma}\label{lem:mlf-stability_est}
For all $\alpha\in(0,1)$,
\begin{equation}\label{eqn:stability_est}
\frac{1}{E_{\alpha,1}(-\llam y^\alpha)}\leq 
1+\Gamma(1-\alpha) \llam y^\alpha\,
\end{equation}
\end{lemma}
\begin{proof}
The bound \eqref{eqn:stability_est} is an immediate consequence of the lower bound in \cite[Theorem 3.25]{BBB}. 
\end{proof}

\begin{lemma}\label{lem:rateE}
For any $\bar{l}>0$,  $\llam_1\geq0$, $\alpha_0\in(0,1)$ and $p\in[1,\frac{1}{1-\alpha_0})$, there exists $\tilde{C}=\tilde{C}(\alpha_0,p,\TT)=\sup_{\alpha'\in[\alpha_0,1)} C(\alpha',p,\TT)>0$ with $C(\alpha',p,\TT)$ as in \eqref{eqn:CalphapTT}, such that for any $\alpha\in[\alpha_0,1)$
and for all $\llam> \llam_1$
\begin{equation}\label{eqn:rateE}
\begin{aligned}
&\left\|d_\alpha\right\|_{L^\infty(0,\TT)}\leq \tilde{C}\llam^{1/p}(1-\alpha),
\hspace*{3cm}\left\| d_\alpha\right\|_{L^p(0,\TT)}\leq \tilde{C}(1-\alpha)\\
&\text{for the function $d_\alpha:[0,\TT]\to\mathbb{R}$ defined by }
d_\alpha(\ttt)=E_{\alpha,1}(-\llam \,\ttt^\alpha)-\exp(-\llam \,\ttt).
\end{aligned}
\end{equation}
\end{lemma}

\begin{proof}
see Appendix~\ref{sec:appendixCauchy1}.
\end{proof}

The estimates from Lemma~\ref{lem:rateE} become more straightforward if the values of $y$ are constrained to a compact interval not containing zero, as relevant for 
Problem~\ref{Cauchy3}. 
This also allows us to derive $L^\infty$ bounds on  $E_{\alpha,1}(-\llam \ttt^\alpha)-\exp(-\llam \ttt)$, which would not be possible without bounding $\ttt$ away from zero, due to the singularities of the Mittag-Leffler functions there.

\begin{lemma}\label{lem:rateEbounds}
For any $0<\ulell<\olell<\infty$, $\llam_1\geq0$, $\alpha_0\in(0,1)$ and $p\in[1,\frac{1}{1-\alpha_0})$, there exist constants $\tilde{C}=\tilde{C}(\alpha_0,p,\olell)>0$ as in Lemma \ref{lem:rateE},  
$\hat{C}=\hat{C}(\alpha_0,\ulell,\olell)>0$, such that for any $\alpha\in[\alpha_0,1)$
and for all $\llam> \llam_1$
\begin{equation}\label{eqn:rateEbounds}
\begin{aligned}
&\left\|d_\alpha\right\|_{L^\infty(\ulell,\olell)}\leq \tilde{C}(1-\alpha),
\hspace*{3cm}\left\|\partial_\ttt d_\alpha\right\|_{L^\infty(\ulell,\olell)}\leq \hat{C}\llam(1-\alpha)\\
&\text{for the function $d_\alpha:[\ulell,\olell]\to\mathbb{R}$ defined by }
d_\alpha(\ttt)=E_{\alpha,1}(-\llam \,\ttt^\alpha)-\exp(-\llam \,\ttt).
\end{aligned}
\end{equation}
\end{lemma}

\begin{proof}
see Appendix~\ref{sec:appendixCauchy1}.
\end{proof}

\medskip

Applying Lemmas~\ref{lem:mlf-stability_est}, \ref{lem:rateE} in \eqref{eqn:totalerror}, \eqref{eqn:approxerr}, \eqref{eqn:propnoise}, we obtain the overall error estimate
\begin{equation}\label{eqn:totalerror1}
\begin{aligned}
\|u-u^{\alpha,\delta}\|_{L^2(0,\olell;H^\sigma(\Omega))} \leq& 
\|u-u^{(\alpha)}\|_{L^2(0,\olell;H^\sigma(\Omega))}\\
&+ \|u_{+0}^\delta-u_{+0}\|_{H^{\sigma}(\Omega)} + \tfrac{\olell^{\alpha+1/2}}{\sqrt{2\alpha+1}}
\Gamma(1-\alpha)\|u_{+0}^\delta-u_{+0}\|_{H^{\sigma+1}(\Omega)},
\end{aligned}
\end{equation}
where we further estimate the approximation error
\[
\begin{aligned}
\|u-u^{(\alpha)}\|_{L^2(0,\olell;H^\sigma(\Omega))}\leq& 
\tilde{C}\bigl((1-\alpha)\| u_+^\dagger\|_{L^2(0,\olell;H^{\sigma+1/p}(\Omega))}\\
&+\olell^\alpha \, (1-\alpha)\Gamma(1-\alpha) \| u_+^\dagger\|_{L^2(0,\olell;H^{\sigma+1+1/p}(\Omega))}\bigr) .
\end{aligned}
\]
Under the assumption $u_+^\dagger\in L^2(0,\olell;H^{\sigma+1+1/p}(\Omega))$, 
from Lebesgue's Dominated Convergence Theorem and uniform boundedness of 
$(1-\alpha)\Gamma(1-\alpha)$  as $\alpha\nearrow1$, as well as convergence to zero of $E_{\alpha,1}(-\llam \,\cdot^\alpha)-\exp(-\llam \,\cdot)$ as $\alpha\nearrow1$, we obtain
$\|u-u^{(\alpha)}\|_{L^2(0,\olell;H^\sigma(\Omega))}\to0$ as $\alpha\nearrow1$.

In view of the fact that the data space in \eqref{eqn:delta} is typically $Y=L^2(\Omega)$, considering the propagated noise term, the $H^\sigma$ and $H^{\sigma+1}$ norms in estimate \eqref{eqn:totalerror1} reveal the fact that even when aiming for the lowest order reconstruction regularity $\sigma=0$, the data needs to be smoothed.

Due to the infinite smoothing property of the forward operator, a method with infinite qualification is required for this purpose. We therefore use Landweber iteration for defining a smoothed version of the data $\uTtd=v^{(i_*)}$ by
\begin{equation}\label{eqn:LW}
v^{(i+1)}=v^{(i)}-A(v^{(i)}-\uTdel)\,, \qquad v^{(0)}=0\,,
\end{equation}
where  
\begin{equation}\label{eqn:Asmoothing}
A=\mu (-\triangle_x)^{-\tilde{\sigma}}
\end{equation}
with $\tilde{\sigma}\in\{\sigma, \sigma+1\}\geq1$ and $\mu>0$ chosen so that $\|A\|_{L^2\to L^2}\leq 1$.

For convergence and convergence rates as the noise level tends to zero, we quote (for the proof, see the appendix of \cite{fracderiv}) a bound in terms of $\|u_+(\cdot,l)\|_{L^2(\Omega)}$ for some fixed $l\in(0,\olell)$, where $u_+$ is the unstable component of the solution according to \eqref{eqn:uplus}.

\begin{lemma}\label{lem:smoothing}
A choice of 
\begin{equation}\label{eqn:istar}
i_*\sim l^{-2} \log\left(\frac{\|u_+(\cdot,l)\|_{L^2(\Omega)}}{\delta}\right)
\end{equation}
yields
\begin{equation}\label{eqn:delta_smoothed}
\begin{aligned}
&\|\uT-\uTtd\|_{L^2(\Omega)}\leq C_1 \delta\,, \\
&\|\uT-\uTtd\|_{H^{\tilde{\sigma}}(\Omega)}\leq C_2 \, l^{-1} \, \delta\, \sqrt{\log\left(\frac{\|u_+(\cdot,l)\|_{L^2(\Omega)}}{\delta}\right)} =:\tilde{\delta}
\end{aligned}
\end{equation}
for some $C_1,C_2>0$ independent of $l$ and $\delta$. 
\end{lemma}

Thus, using $\uTtd$ in place of $\uT$ in the reconstruction, we obtain the following convergence result.

\begin{theorem}
Let the exact solution $u^\dagger$ of Problem~\ref{Cauchy1} satisfy $u_+^\dagger\in L^2(0,\olell;H^{\sigma+1+1/p}(\Omega))$ for some $\sigma\geq0$, $p>1$ and let the noisy data satisfy \eqref{eqn:delta} with smoothed data constructed as in Lemma~\ref{lem:smoothing}. Further, assume that $\alpha=\alpha(\tilde{\delta})$ is chosen such that $\alpha(\tilde{\delta})\to1$ and 
$\Gamma(1-\alpha(\tilde{\delta}))\tilde{\delta}
\to0$ as $\delta\to0$. Then $\|u-u^{\alpha(\tilde{\delta}),\delta}\|_{L^2(0,\olell;H^\sigma(\Omega))}\to0$ as $\delta\to0$.
\end{theorem}

Since $\Gamma(1-\alpha)\sim(1-\alpha)^{-1}$ as $\alpha\nearrow1$, the condition $\Gamma(1-\alpha(\tilde{\delta}))\tilde{\delta}\to0$ means that $\alpha(\tilde{\delta})$ must not converge to unity too fast as the noise level vanishes -- a well known condition in the context of regularisation of ill-posed problems.

\section{Problem 2}\label{sec:Cauchy2}
\begin{problem}\label{Cauchy2}
Given $f$, $g$ in
\[
\begin{aligned}
&-\triangle u =0  \mbox{ in }x\in(0,L)\,, \ y\in(0,\ell(x))\\
&B_0u(0,y)=0\,, \ x\in(0,\ell(0))\quad B_Lu(L,y)=0\,, \ x\in(0,\ell(L))\\
&u(x,0)=f(x)\,, u_y(x,0)=g(x)\, \quad x\in\Omega
\end{aligned}
\]
find $\ell:(0,w)\to(0,\bar{l})$ such that one of the following three conditions holds on the interface defined by $\ell$.
\begin{equation}\label{BCell}
\begin{aligned}
&(N)\quad \neumann{\partial_\nu u=0 \ : \quad  0=\partial_\nnu u(x,\ell(x))=u_y(x,\ell(x))-\ell'(x) u_x(x,\ell(x)) 
\quad x\in(0,L)}\text{ or }\\
&(D)\quad \dirichlet{u=0 \ : \quad u(x,\ell(x))\quad x\in(0,L)} \text{ or }\\
&(I)\quad \impedance{\partial_\nu u+\gamma u = 0 \ : \quad  0=\partial_\nnu u(x,\ell(x)) + \sqrt{1+\ell'(x)^2}\gamma(x)u(x,\ell(x))
\quad x\in(0,L)}.
\end{aligned}
\end{equation}
\end{problem}
(Note that in \eqref{BCell} $\nnu$ is the non-normalised outer normal direction, but for the zero level set this does not matter.)

A by now classical reference for this problem is \cite{Alessandrini:1997}.

Here we make the a priori assumption that the searched for domain is contained in the rectangular box $(0,L)\times(0,\olell)$.\hfill\break
The operators $B_0$ and $B_L$ determine the boundary conditions on the lateral boundary parts, which may also be of Dirichlet, Neumann, or impedance type.

To emphasise dependence on the parametrisation $\ell$, we denote the domain as well as the fixed and variable parts of its boundary as follows
\begin{equation}\label{Dofell}
\begin{aligned}
&D(\ell)=\{(x,y)\in(0,L)\times(0,\olell)\, : \,  y\in(0,\ell(x))\}\,, \\
&\Gamma_0(\ell)=\{(x,\ell(x))\, : \, x\in(0,L)\}\\
&\Gamma_1=(0,L)\times\{0\}\,,  \\ 
&\Gamma_2(\ell)=\{0\}\times(0,\ell(0))\cup \{L\}\times(0,\ell(L))\\
\end{aligned}
\end{equation}
(note that $\Gamma_2(\ell)$ depends on $\ell$ only weakly via its endpoints $\ell(0)$ and $\ell(L)$).
With this we can write the forward problem as
\begin{equation}\label{forward}
\begin{aligned}
&\hspace*{0.4cm}-\triangle u =0  \mbox{ in }D(\ell)\\
&\hspace*{1.3cm}u=f\mbox{ on }\Gamma_1\\
&\hspace*{1cm}Bu=0\mbox{ on }\Gamma_2(\ell)\\
&\left.\begin{array}{lll}
(N)& \neumann{\partial_\nnu u=0 \ } &\ \text{or}\\
(D)& \dirichlet{u=0 } &\ \text{or}\\
(I)& \impedance{\partial_\nu u+\gamma u =0 \ }.
\end{array}\right\}
\mbox{ on }\Gamma_0(\ell).
\end{aligned}
\end{equation}

\medskip

We will split the full inverse problem into two subproblems: The linear severely ill-posed Cauchy problem  on $D(\olell)$ (which is our Problem~\ref{Cauchy1}) and a well posed (or maybe mildly ill-posed) nonlinear problem of reconstructing the curve $\ell$.
Thus a straightforward approach for solving Problem~\ref{Cauchy2} would be to first solve a regularised version of Problem~\ref{Cauchy1} (e.g., in the way devised in Section ~\ref{sec:Cauchy1}) and then applying Newton's method to recover $\ell$ from \eqref{BCell}.

\medskip

However, we follow a combined approach, writing Problem~\ref{Cauchy2} as an operator equation with the total forward operator $F$, and applying a Newton scheme, in which we make use of a regularised solution of Problem~\ref{Cauchy1}.

The forward operator is defined by
\[
F:\ell\mapsto u_y(x,0)-g(x)\mbox{ where $u$ solves \eqref{forward}} .
\]
Its linearisation in the direction $\dl:(0,w)\to\mathbb{R}$ of the parametrisation $\ell$ is 
is given by $F'(\ell)\dl =\partial_y v\vert_{\Gamma_1}$, where $v$ solves
\begin{equation}\label{FprimeCauchy2}
\begin{aligned}
&\hspace*{0.4cm}-\triangle v =0  \mbox{ in }D(\ell)\\
&\hspace*{1.3cm}v=0\mbox{ on }\Gamma_1\\
&\hspace*{1cm}Bv=0\mbox{ on }\Gamma_2(\ell)\\
&\left.\begin{array}{lll}
(N)&\ \neumann{\partial_\nnu v(x,\ell(x))= \tfrac{d}{dx} \left(\dl (x)u_x(x,\ell(x))\right)
}
&\text{ or}\\
(D)&\ \dirichlet{v(x,\ell(x))=-u_y(x,\ell(x)) \dl (x)} 
&\text{ or}\\
(I)&\ \impedance{\partial_\nnu v(x,\ell(x))+\sqrt{1+\ell'(x)^2}\gamma(x)v(x,\ell(x))=
G(u,\ell)\dl }
\end{array}\right\}
\mbox{ on }\Gamma_0(\ell).
\end{aligned}
\end{equation}
where 
in the Neumann case (N) we can also use the alternative formulation
$\partial_\nnu v(x,\ell(x))= \tfrac{d}{dx} \left(\dl (x)u_x(x,\ell(x))\right)
= \tfrac{d}{dx} \left(\dl (x)\tfrac{1}{1+\ell'(x)^2} \partial_{\ttau}u(x,\ell(x))\right)$ and in the impedance case the boundary value function $G$ is defined by
\impedance{
\[
\begin{aligned}
G(u,\ell)\dl :=&
\dl'(x)\Bigl(u_x(x,\ell(x))-\tfrac{\ell'(x)}{\sqrt{1+\ell'(x)^2}}\gamma(x)u(x,\ell(x))\Bigr)\\
&\quad-\dl (x)\Bigl(u_{yy}(x,\ell(x))-\ell'(x)u_{xy}(x,\ell(x))
+\sqrt{1+\ell'(x)^2}\gamma(x)u_y(x,\ell(x))\Bigr)\\
=:&\tfrac{d}{dx}[\coeffalpha[\ell,u] \dl]-\coeffbeta[\ell,u];
\end{aligned}
\]
see see Appendix~\ref{sec:appendixCauchy2} for more explicit formulas for $\coeffalpha[\ell,u]$, $\coeffbeta[\ell,u]$.}
Note that this is obtained in a similar manner to the formula for the shape derivative \cite[equation (3.1)]{Rundell2008a}, see also \cite{Hettlich1995}, and the identity $ds=\sqrt{1+\ell'(x)^2}\, dx$ for the arclength parametrisation $s$, but using $\nnu$ in place of $\nu$, as well as 
\[
\begin{aligned}
\vec{x}_h(x)=\left(\begin{array}{c}0\\ \dl(x)\end{array}\right)\,,\quad 
&\vec{\nnu}(x)= \left(\begin{array}{c}-\ell'(x)\\1\end{array}\right)\,,\quad 
\vec{\nu}(x)=\frac{1}{\sqrt{1+\ell'(x)^2}} \vec{\nnu}(x)\\
&\vec{\ttau}(x)= \left(\begin{array}{c}1\\ \ell'(x)\end{array}\right)\,,\quad 
\vec{\tau}(x)=\frac{1}{\sqrt{1+\ell'(x)^2}} \vec{\ttau}(x).
\end{aligned}
\]

Thus, computation of a Newton step $\dl =\dl^{(k)}$ starting from some iterate $\ell^{(k)}=\ell$ amounts to solving the system
\begin{equation}\label{eqn:Newtonstep}
\begin{aligned}
&\hspace*{0.4cm}-\triangle z =0  \mbox{ in }D(\ell)\\
&\hspace*{1.3cm}z=f\mbox{ on }\Gamma_1\\
&\hspace*{1cm}Bz=0\mbox{ on }\Gamma_2(\ell)\\
&\hspace*{1.2cm}z_y=g\mbox{ on }\Gamma_1\\
&\left.\begin{array}{lll}
(N)&\ \neumann{\partial_\nnu z(x,\ell(x))= \tfrac{d}{dx} \left(\dl (x)u_x(x,\ell(x))\right)}
&\text{ or}\\
(D)&\ \dirichlet{z(x,\ell(x))=-u_y(x,\ell(x)) \dl(x)} 
&\text{ or}\\
(I)&\ \impedance{\partial_\nnu z(x,\ell(x))+\sqrt{1+\ell'(x)^2}\gamma(x)z(x,\ell(x))=
\frac{d}{dx}[\coeffalpha[\ell,u] \dl]-\coeffbeta[\ell,u]}
\end{array}\right\}
\mbox{ on }\Gamma_0(\ell).
\end{aligned}
\end{equation}
(note that $z_y\vert_{\Gamma_1}-g=F(\ell)+F'(\ell)\dl\;$). \hfill\break

If we solve (a regularised version of) the Cauchy problem on the rectangular hold-all domain
\begin{equation}\label{eqn:zbar}
\begin{aligned}
&-\triangle \bar{z} =0  \mbox{ in }D(\olell)=(0,L)\times(0,\olell)\\
&\bar{z}=f\mbox{ on }\Gamma_1\\
&B\bar{z}=0\mbox{ on }\Gamma_2(\olell)=\{0,L\}\times(0,\olell)\\
&\partial_y \bar{z}-g=0\mbox{ on }\Gamma_1\\
\end{aligned}
\end{equation}
in advance, then by uniqueness of solutions to the Cauchy problem, $z$ coincides with $\bar{z}$ on $D(\ell)$.

Therefore, in each Newton step it only remains to compute $u^{(k)}=u$ from the well-posed mixed elliptic boundary value problem \eqref{forward} with $\ell=\ell^{(k)}$ and update as follows.

\neumann{
\underline{In the Neumann case},
\begin{equation}\label{eqn:NewtonstepNeumann}
\begin{aligned}
(N)\ \ \ell^{(k+1)}(x)&=\ell^{(k)}(x)+
\frac{1}{u_x^{(k)}(x,\ell^{(k)}(x))}\int_0^x\partial_\nnu \bar{z}(\xi,\ell^{(k)}(\xi))\, d\xi \\
&=\ell^{(k)}(x)+
\frac{1}{u_x^{(k)}(x,\ell^{(k)}(x))}\int_0^x \Bigl(\partial_y \bar{z}(\xi,\ell^{(k)}(\xi))-{\ell^{(k)}}'(\xi)\partial_x \bar{z}(\xi,\ell^{(k)}(\xi))\Bigr)\, d\xi.
\end{aligned}
\end{equation}
If $B$ denotes the lateral Dirichlet trace, then this also needs to be regularised, since 
due to the identity $u(0,y)=0=u(L,y)$, the partial derivative  $u_x(\cdot,y)$ has to vanish at least at one interior point $x$ for each $y$.
To avoid problems arising from division by zero, 
we thus solve a regularised version
\[
\begin{aligned}
(\ell^{(k+1)}-\ell^{(k)})
&=\textup{argmin}_{\dl} \int_0^L
\Bigl(\partial_\nnu z(x,\ell(x)) - \tfrac{d}{dx} [\dl(x)u_x(x,\ell(x))]\Bigr)^2\, dx\\
&\hspace*{5cm}+\frac{1}{\rho_1}\int_0^L \dl'(x)^2\, dx +\rho_2 (\dl(0)^2+\dl(L)^2)
\end{aligned}
\]
with a regularisation parameter $\frac{1}{\rho_1}$ and a penalisation parameter $\rho_2$ enforcing our assumption of $\dl$ to be known at the boundary points.
}

\medskip

\dirichlet{
\underline{In the Dirichlet case}, the Newton step computes as 
\begin{equation}\label{eqn:NewtonstepDirichlet}
\hspace*{-1cm}
(D)\ \ \ell^{(k+1)}(x)=\ell^{(k)}(x)-\frac{\bar{z}(x,\ell^{(k)}(x))}{u_y^{(k)}(x,\ell^{(k)}(x))}.
\hspace*{7cm}
\end{equation}
With lateral Dirichlet conditions $Bu=u=0$ we have $u_y^{(k)}(0,\ell(0))=u_y^{(k)}(L,\ell(L))=0$ and so would have to divide by numbers close to zero near the endpoints. This can be avoided by imposing Neumann conditions $Bu=\partial_\nu u=0$ on the lateral boundary.
Still, the problem is mildly ill-posed and thus needs to be regularised for the following reason.
In view of the Implicit Function Theorem, the function $\ell$, being implicitly defined by $u(x,\ell(x))=0$, has the same order of differentiability as $u$.
However, \eqref{eqn:NewtonstepDirichlet} contains an additional derivative of $u$ as compared to $\ell$. Obtaining a bound on $u_y^{(k)}$ in terms of $\ell^{(k)}$ from elliptic regularity, (cf., e.g., \cite{grisvard_2011}) cannot be expected to be possible with the same level of differentiability.
}

\medskip

\impedance{
\underline{In the impedance case}, with 
\begin{equation*}
\begin{aligned}
\coeffalpha[\ell,u](x)&=
\begin{cases}
\tfrac{d}{dx} u(x,\ell(x))=u_x(x,\ell(x)) &\text{ if } \ell'(x)=0\\
\tfrac{1}{\ell'(x)}\Bigl(\tfrac{\gamma(x)}{\sqrt{1+\ell'(x)^2}}u(x,\ell(x))+u_y(x,\ell(x)))\Bigr)
&\text{ otherwise}\end{cases}\\[1ex]
\coeffbeta[\ell,u](x)&=\Bigl(-\sqrt{1+\ell'(x)^2}\gamma(x)u_y(x,\ell(x))+\tfrac{d}{dx}\Bigl[\tfrac{\ell'(x)}{\sqrt{1+\ell'(x)^2}}\gamma(x)u(x,\ell(x))\Bigr]\Bigr)\\
&=\Bigl(\tfrac{\ell''(x)}{\sqrt{1+\ell'(x)^2}^3}\gamma(x)
+\tfrac{\ell'(x)}{\sqrt{1+\ell'(x)^2}} \gamma'(x) +\gamma(x)^2\Bigr) u(x,\ell(x))
\end{aligned}
\end{equation*}
and 
\begin{equation*}
\begin{aligned}
\phi(x)&= \dl(x)\, \coeffalpha[\ell,u](x)
\qquad
a(x)=
\tfrac{\coeffbeta[\ell,u]}{\coeffalpha[\ell,u]}(x),
\\
b(x)&=\partial_\nnu \bar{z}(x,\ell^{(k)}(x)) + \sqrt{1+\ell'(x)^2}\gamma(x)\bar{z}(x,\ell^{(k)}(x))
\end{aligned}
\end{equation*}
the Newton step amounts to solving $\tfrac{d}{dx} \phi(x) -a(x)\phi(x)=b(x)$, which yields
\begin{equation}\label{eqn:NewtonstepImpedance}
\begin{aligned}
(I)\ \ \ell^{(k+1)}(x)&=\ell^{(k)}(x)-\tfrac{1}{\coeffalpha[\ell^{(k)},u^{(k)}](x)}
\Bigl\{
\exp\Bigl(-\int_0^x\tfrac{\coeffbeta[\ell^{(k)},u^{(k)}]}{\coeffalpha[\ell^{(k)},u^{(k)}]}(s)\, ds\Bigr) \dl(0)\coeffalpha[\ell^{(k)},u^{(k)}](0)\\
&\hspace*{5cm}+\int_0^x b(s) \exp\Bigl(-\int_s^x\tfrac{\coeffbeta[\ell^{(k)},u^{(k)}]}{\coeffalpha[\ell^{(k)},u^{(k)}]}(t)\, dt\Bigr)\, ds
\Bigr\}.
\end{aligned}
\end{equation}
See Appendix~\ref{sec:appendixCauchy2} for details on the derivation of this formula.
Also here, due to the appearance of derivatives of $u$ and $\ell$, regularisation is needed.
}

\hfill\break
In Section~\ref{sec:Cauchy3}, we will prove convergence of a regularised frozen Newton method for simultaneously recovering $\ell$ and $\gamma$.

\begin{remark}(Uniqueness)
In the Neumann case $\partial_\nnu u=0$ on $\Gamma_0(\ell)$, the linearisation $F'(\ell)$ is not injective since $F'(\ell)\dl=0$ only implies that $\dl(x)u_x(x,\ell(x))$ is constant.\hfill\break
There is nonuniqueness in the nonlinear inverse problem $F(\ell)=0$ as well, as the counterexample $f(x)=\sin(\pi x/L)$, $g(x)=0$, $u(x,y)=f(x)$ shows;
all horizontal lines $\ell(x)\equiv c$ for $c\in\mathbb{R}^+$ solve the inverse problem.

\medskip

\dirichlet{
In the Dirichlet case $u=0$ on $\Gamma_0(\ell)$, linearised uniqueness follows from the formula $z(x,\ell(x))=-u_y(x,\ell(x)) \dl(x)$ provided $u_y$ does not vanish on an open subset $\Gamma$ of $\Gamma_0(\ell)$. The latter can be excluded by Holmgren's theorem, since $-\triangle u=0$, together with the conditions $u=0$, $u_y=0$ on $\Gamma_0(\ell)$ defines a noncharacteristic Cauchy problem and therefore would imply $u\equiv0$ on $D(\ell)$, a contradiction to $f\not=0$.\hfill\break
Full uniqueness can be seen from the fact that if $\ell$ and $\tilde{\ell}$ solve the inverse problem, then on the domain enclosed by these two curves (plus possibly some $\Gamma_2$ boundary part), $u$ satisfies a homogeneous Dirichlet Laplace problem and therefore has to vanish identically. This, on the other hand would yield a homogeneous Cauchy problem for $u$ on the part $D(\min\{\ell,\tilde{\ell}\})$ that lies below both curves and thus imply that $u$ vanishes identically there. Again we would then have a contradiction to $f\not=0$. \\
This uniqueness proof would also work with Neumann or impedance instead of Dirichlet conditions on the lateral boundary $\Gamma_2$.
}

\medskip

\impedance{
For uniqueness in the impedance case, see also \cite[Theorem 2.2]{KressRundell:2001}.
}
\end{remark}

\subsection{Reconstructions}\label{sec:reconCauchy2}
\input colordvi

\input pictex
\font\smallsymbol = cmmi8
\newdimen\xfiglen \newdimen\yfiglen
\xfiglen=4 true in
\yfiglen=2 true in
\newbox\figurelegendone
\newbox\figurelegendtwo
\newbox\figurelegendthree
\newbox\figureone
\newbox\figuretwo
\newbox\figurethree
\newbox\figurefour
\newbox\figurefive
\newbox\figuresix

%
\setbox\figurelegendtwo=\hbox{
\beginpicture
  \setcoordinatesystem units <0.06\xfiglen,0.05\yfiglen> 
  \setplotarea x from 0 to 2, y from 0 to 6
\footnotesize
\linethickness=0.5pt
\scriptsize
   \setdashes <3pt>  \putrule from 0 3 to 1.25 3
   \setsolid
   \Black{\relax \putrule from 0 6 to 1 6 \relax}\relax
   \put {$\ell(x)$ actual}  [l] at 1.2 6
   \Red{\relax \putrule from 0 5 to 1 5 \relax}\relax
   \put {iter 5 } [l] at 1.2 5
   \Purple{\relax \putrule from 0 4 to 1 4 \relax}\relax
   \put {iter 4 } [l] at 1.2 4
   \Blue{\relax \putrule from 0 3 to 1 3 \relax}\relax
   \put {iter 3 } [l] at 1.2 3
   \Cyan{\relax \putrule from 0 2 to 1 2 \relax}\relax
   \put {iter 2 } [l] at 1.2 2
   \Green{\relax \putrule from 0 1 to 1 1 \relax}\relax
   \put {iter 1 } [l] at 1.2 1
\setdots <3pt>
   \Black{\relax \putrule from 0 0 to 1 0 \relax}\relax
   \put {$\ell(x)$ initial } [l] at 1.2 -0.1
\endpicture
}
\setbox\figuretwo=\vbox{\hsize=\xfiglen
\beginpicture
  \setcoordinatesystem units <0.4\xfiglen,10\yfiglen>  point at 0.0 0.0
  \setplotarea x from 0 to 1, y from 0.0 to 0.1
\relax
\scriptsize
  \axis bottom ticks short numbered from 0 to 1 by 0.2 /
  \axis left ticks short numbered from 0 to 0.09 by 0.01 /
\footnotesize
\put {$\ell(x)$} [l] at 0.02 0.1
\linethickness=0.6pt
\putrule from 1 0.0 to 1 0.1
\setdashes <3pt>
\putrule from 0 0.02 to 0.7 0.02
\put {$\ell_{init}$} at 0.75 0.02
\putrule from 0.8 0.02 to 1 0.02
\setquadratic
\setsolid
\Black{\relax
\plot    
  0.0000   0.0900
  0.0200   0.0899
  0.0400   0.0897
  0.0600   0.0893
  0.0800   0.0888
  0.1000   0.0881
  0.1200   0.0873
  0.1400   0.0864
  0.1600   0.0854
  0.1800   0.0843
  0.2000   0.0831
  0.2200   0.0819
  0.2400   0.0806
  0.2600   0.0794
  0.2800   0.0781
  0.3000   0.0769
  0.3200   0.0757
  0.3400   0.0746
  0.3600   0.0736
  0.3800   0.0727
  0.4000   0.0719
  0.4200   0.0712
  0.4400   0.0707
  0.4600   0.0703
  0.4800   0.0701
  0.5000   0.0700
  0.5200   0.0701
  0.5400   0.0703
  0.5600   0.0707
  0.5800   0.0712
  0.6000   0.0719
  0.6200   0.0727
  0.6400   0.0736
  0.6600   0.0746
  0.6800   0.0757
  0.7000   0.0769
  0.7200   0.0781
  0.7400   0.0794
  0.7600   0.0806
  0.7800   0.0819
  0.8000   0.0831
  0.8200   0.0843
  0.8400   0.0854
  0.8600   0.0864
  0.8800   0.0873
  0.9000   0.0881
  0.9200   0.0888
  0.9400   0.0893
  0.9600   0.0897
  0.9800   0.0899
  1.0000   0.0900
 /\relax}\relax
\linethickness=0.4pt
\setdashes <3pt>
\setsolid
\Green{\relax   
\plot
  0.0000   0.0355
  0.0200   0.0355
  0.0400   0.0355
  0.0600   0.0355
  0.0800   0.0355
  0.1000   0.0354
  0.1200   0.0354
  0.1400   0.0353
  0.1600   0.0353
  0.1800   0.0352
  0.2000   0.0352
  0.2200   0.0351
  0.2400   0.0350
  0.2600   0.0350
  0.2800   0.0349
  0.3000   0.0348
  0.3200   0.0347
  0.3400   0.0346
  0.3600   0.0346
  0.3800   0.0345
  0.4000   0.0344
  0.4200   0.0344
  0.4400   0.0343
  0.4600   0.0343
  0.4800   0.0343
  0.5000   0.0343
  0.5200   0.0343
  0.5400   0.0343
  0.5600   0.0344
  0.5800   0.0344
  0.6000   0.0345
  0.6200   0.0345
  0.6400   0.0346
  0.6600   0.0347
  0.6800   0.0348
  0.7000   0.0348
  0.7200   0.0349
  0.7400   0.0350
  0.7600   0.0350
  0.7800   0.0351
  0.8000   0.0352
  0.8200   0.0352
  0.8400   0.0353
  0.8600   0.0353
  0.8800   0.0354
  0.9000   0.0354
  0.9200   0.0355
  0.9400   0.0355
  0.9600   0.0355
  0.9800   0.0355
  1.0000   0.0355
 /\relax}\relax
\setsolid
\Cyan{\relax   
\plot
  0.0000   0.0569
  0.0200   0.0569
  0.0400   0.0568
  0.0600   0.0568
  0.0800   0.0567
  0.1000   0.0565
  0.1200   0.0564
  0.1400   0.0562
  0.1600   0.0560
  0.1800   0.0557
  0.2000   0.0555
  0.2200   0.0552
  0.2400   0.0549
  0.2600   0.0545
  0.2800   0.0542
  0.3000   0.0538
  0.3200   0.0535
  0.3400   0.0532
  0.3600   0.0529
  0.3800   0.0526
  0.4000   0.0523
  0.4200   0.0521
  0.4400   0.0520
  0.4600   0.0519
  0.4800   0.0518
  0.5000   0.0518
  0.5200   0.0519
  0.5400   0.0520
  0.5600   0.0522
  0.5800   0.0524
  0.6000   0.0526
  0.6200   0.0528
  0.6400   0.0531
  0.6600   0.0534
  0.6800   0.0537
  0.7000   0.0540
  0.7200   0.0543
  0.7400   0.0546
  0.7600   0.0548
  0.7800   0.0551
  0.8000   0.0554
  0.8200   0.0556
  0.8400   0.0559
  0.8600   0.0561
  0.8800   0.0563
  0.9000   0.0565
  0.9200   0.0567
  0.9400   0.0568
  0.9600   0.0569
  0.9800   0.0570
  1.0000   0.0570
/\relax}\relax
\Blue{\relax   
\plot
  0.0000   0.0775
  0.0200   0.0775
  0.0400   0.0774
  0.0600   0.0772
  0.0800   0.0769
  0.1000   0.0766
  0.1200   0.0762
  0.1400   0.0757
  0.1600   0.0752
  0.1800   0.0746
  0.2000   0.0739
  0.2200   0.0732
  0.2400   0.0724
  0.2600   0.0716
  0.2800   0.0708
  0.3000   0.0700
  0.3200   0.0692
  0.3400   0.0684
  0.3600   0.0677
  0.3800   0.0670
  0.4000   0.0665
  0.4200   0.0660
  0.4400   0.0657
  0.4600   0.0654
  0.4800   0.0653
  0.5000   0.0653
  0.5200   0.0655
  0.5400   0.0657
  0.5600   0.0661
  0.5800   0.0665
  0.6000   0.0670
  0.6200   0.0676
  0.6400   0.0682
  0.6600   0.0689
  0.6800   0.0696
  0.7000   0.0703
  0.7200   0.0710
  0.7400   0.0717
  0.7600   0.0724
  0.7800   0.0730
  0.8000   0.0737
  0.8200   0.0744
  0.8400   0.0750
  0.8600   0.0756
  0.8800   0.0761
  0.9000   0.0766
  0.9200   0.0770
  0.9400   0.0774
  0.9600   0.0776
  0.9800   0.0778
  1.0000   0.0778
/\relax}\relax
\Purple{\relax   
\plot
  0.0000   0.0878
  0.0200   0.0878
  0.0400   0.0876
  0.0600   0.0873
  0.0800   0.0869
  0.1000   0.0864
  0.1200   0.0857
  0.1400   0.0850
  0.1600   0.0842
  0.1800   0.0832
  0.2000   0.0822
  0.2200   0.0811
  0.2400   0.0799
  0.2600   0.0787
  0.2800   0.0775
  0.3000   0.0763
  0.3200   0.0751
  0.3400   0.0740
  0.3600   0.0729
  0.3800   0.0720
  0.4000   0.0712
  0.4200   0.0706
  0.4400   0.0701
  0.4600   0.0698
  0.4800   0.0696
  0.5000   0.0697
  0.5200   0.0698
  0.5400   0.0702
  0.5600   0.0707
  0.5800   0.0713
  0.6000   0.0720
  0.6200   0.0729
  0.6400   0.0738
  0.6600   0.0747
  0.6800   0.0757
  0.7000   0.0767
  0.7200   0.0777
  0.7400   0.0788
  0.7600   0.0798
  0.7800   0.0808
  0.8000   0.0818
  0.8200   0.0828
  0.8400   0.0838
  0.8600   0.0847
  0.8800   0.0856
  0.9000   0.0864
  0.9200   0.0870
  0.9400   0.0876
  0.9600   0.0880
  0.9800   0.0883
  1.0000   0.0883
/\relax}\relax
\Red{\relax   
\plot
  0.0000   0.0895
  0.0200   0.0895
  0.0400   0.0892
  0.0600   0.0889
  0.0800   0.0885
  0.1000   0.0879
  0.1200   0.0872
  0.1400   0.0864
  0.1600   0.0854
  0.1800   0.0844
  0.2000   0.0833
  0.2200   0.0821
  0.2400   0.0808
  0.2600   0.0795
  0.2800   0.0782
  0.3000   0.0769
  0.3200   0.0756
  0.3400   0.0744
  0.3600   0.0733
  0.3800   0.0724
  0.4000   0.0715
  0.4200   0.0709
  0.4400   0.0704
  0.4600   0.0700
  0.4800   0.0699
  0.5000   0.0699
  0.5200   0.0701
  0.5400   0.0705
  0.5600   0.0710
  0.5800   0.0716
  0.6000   0.0724
  0.6200   0.0733
  0.6400   0.0742
  0.6600   0.0752
  0.6800   0.0763
  0.7000   0.0774
  0.7200   0.0784
  0.7400   0.0796
  0.7600   0.0807
  0.7800   0.0818
  0.8000   0.0829
  0.8200   0.0839
  0.8400   0.0850
  0.8600   0.0860
  0.8800   0.0870
  0.9000   0.0879
  0.9200   0.0886
  0.9400   0.0893
  0.9600   0.0897
  0.9800   0.0901
  1.0000   0.0901
/\relax}\relax
\endpicture
}
%

%
\setbox\figurelegendthree=\hbox{
\beginpicture
  \setcoordinatesystem units <0.06\xfiglen,0.05\yfiglen> 
  \setplotarea x from 0 to 2, y from 0 to 6
\footnotesize
\linethickness=0.5pt
\scriptsize
   \setdashes <3pt>  \putrule from 0 3 to 1.25 3
   \setsolid
   \Black{\relax \putrule from 0 6 to 1 6 \relax}\relax
   \put {$\ell(x)$ actual}  [l] at 1.2 6
   \Cyan{\relax \putrule from 0 2 to 1 2 \relax}\relax
   \put {iter 2 } [l] at 1.2 2
   \Green{\relax \putrule from 0 1 to 1 1 \relax}\relax
   \put {iter 1 } [l] at 1.2 1
\setdots <3pt>
   \Black{\relax \putrule from 0 0 to 1 0 \relax}\relax
   \put {$\ell(x)$ initial } [l] at 1.2 -0.1
\endpicture
}
\setbox\figurethree=\vbox{\hsize=\xfiglen
\beginpicture
  \setcoordinatesystem units <0.4\xfiglen,10\yfiglen>  point at 0.0 0.0
  \setplotarea x from 0 to 1, y from 0.0 to 0.1
\relax
\scriptsize
  \axis bottom ticks short numbered from 0 to 1 by 0.2 /
  \axis left ticks short numbered from 0 to 0.09 by 0.01 /
\footnotesize
\put {$\ell(x)$} [l] at 0.02 0.1
\linethickness=0.6pt
\putrule from 1 0.0 to 1 0.1
\setdashes <3pt>
\putrule from 0 0.02 to 0.7 0.02
\put {$\ell_{init}$} at 0.75 0.02
\putrule from 0.8 0.02 to 1 0.02
\setquadratic
\setsolid
\Black{\relax
\plot    
  0.0000   0.0900
  0.0200   0.0899
  0.0400   0.0897
  0.0600   0.0893
  0.0800   0.0888
  0.1000   0.0881
  0.1200   0.0873
  0.1400   0.0864
  0.1600   0.0854
  0.1800   0.0843
  0.2000   0.0831
  0.2200   0.0819
  0.2400   0.0806
  0.2600   0.0794
  0.2800   0.0781
  0.3000   0.0769
  0.3200   0.0757
  0.3400   0.0746
  0.3600   0.0736
  0.3800   0.0727
  0.4000   0.0719
  0.4200   0.0712
  0.4400   0.0707
  0.4600   0.0703
  0.4800   0.0701
  0.5000   0.0700
  0.5200   0.0701
  0.5400   0.0703
  0.5600   0.0707
  0.5800   0.0712
  0.6000   0.0719
  0.6200   0.0727
  0.6400   0.0736
  0.6600   0.0746
  0.6800   0.0757
  0.7000   0.0769
  0.7200   0.0781
  0.7400   0.0794
  0.7600   0.0806
  0.7800   0.0819
  0.8000   0.0831
  0.8200   0.0843
  0.8400   0.0854
  0.8600   0.0864
  0.8800   0.0873
  0.9000   0.0881
  0.9200   0.0888
  0.9400   0.0893
  0.9600   0.0897
  0.9800   0.0899
  1.0000   0.0900
 /\relax}\relax
\linethickness=0.4pt
\setdashes <3pt>
\setsolid
\Green{\relax   
\plot
  0.0000   0.0849
  0.0200   0.0848
  0.0400   0.0846
  0.0600   0.0844
  0.0800   0.0840
  0.1000   0.0835
  0.1200   0.0829
  0.1400   0.0823
  0.1600   0.0816
  0.1800   0.0808
  0.2000   0.0799
  0.2200   0.0790
  0.2400   0.0781
  0.2600   0.0772
  0.2800   0.0763
  0.3000   0.0753
  0.3200   0.0744
  0.3400   0.0736
  0.3600   0.0728
  0.3800   0.0721
  0.4000   0.0715
  0.4200   0.0709
  0.4400   0.0705
  0.4600   0.0702
  0.4800   0.0700
  0.5000   0.0699
  0.5200   0.0700
  0.5400   0.0702
  0.5600   0.0705
  0.5800   0.0709
  0.6000   0.0715
  0.6200   0.0721
  0.6400   0.0728
  0.6600   0.0736
  0.6800   0.0745
  0.7000   0.0754
  0.7200   0.0763
  0.7400   0.0773
  0.7600   0.0782
  0.7800   0.0791
  0.8000   0.0800
  0.8200   0.0808
  0.8400   0.0816
  0.8600   0.0823
  0.8800   0.0829
  0.9000   0.0835
  0.9200   0.0839
  0.9400   0.0843
  0.9600   0.0845
  0.9800   0.0847
  1.0000   0.0847
 /\relax}\relax
\setsolid
\Cyan{\relax   
\plot
  0.0000   0.0900
  0.0200   0.0899
  0.0400   0.0897
  0.0600   0.0893
  0.0800   0.0887
  0.1000   0.0880
  0.1200   0.0872
  0.1400   0.0863
  0.1600   0.0853
  0.1800   0.0842
  0.2000   0.0830
  0.2200   0.0818
  0.2400   0.0806
  0.2600   0.0793
  0.2800   0.0781
  0.3000   0.0769
  0.3200   0.0758
  0.3400   0.0747
  0.3600   0.0737
  0.3800   0.0728
  0.4000   0.0720
  0.4200   0.0713
  0.4400   0.0708
  0.4600   0.0704
  0.4800   0.0702
  0.5000   0.0701
  0.5200   0.0702
  0.5400   0.0704
  0.5600   0.0708
  0.5800   0.0713
  0.6000   0.0720
  0.6200   0.0728
  0.6400   0.0737
  0.6600   0.0747
  0.6800   0.0758
  0.7000   0.0770
  0.7200   0.0782
  0.7400   0.0795
  0.7600   0.0807
  0.7800   0.0820
  0.8000   0.0832
  0.8200   0.0843
  0.8400   0.0853
  0.8600   0.0863
  0.8800   0.0871
  0.9000   0.0879
  0.9200   0.0885
  0.9400   0.0890
  0.9600   0.0893
  0.9800   0.0895
  1.0000   0.0893
/\relax}\relax
\endpicture
}
%

%
\setbox\figurelegendthree=\hbox{
\beginpicture
  \setcoordinatesystem units <0.06\xfiglen,0.05\yfiglen> 
  \setplotarea x from 0 to 2, y from 0 to 6
\footnotesize
\linethickness=0.5pt
\scriptsize
   \setdashes <3pt>  \putrule from 0 3 to 1.25 3
   \setsolid
   \Black{\relax \putrule from 0 6 to 1 6 \relax}\relax
   \put {$\ell(x)$ actual}  [l] at 1.2 6
   \Cyan{\relax \putrule from 0 2 to 1 2 \relax}\relax
   \put {iter 2 } [l] at 1.2 2
   \Green{\relax \putrule from 0 1 to 1 1 \relax}\relax
   \put {iter 1 } [l] at 1.2 1
\setdots <3pt>
   \Black{\relax \putrule from 0 0 to 1 0 \relax}\relax
   \put {$\ell(x)$ initial } [l] at 1.2 -0.1
\endpicture
}
\setbox\figurefour=\vbox{\hsize=\xfiglen
\beginpicture
  \setcoordinatesystem units <0.4\xfiglen,10\yfiglen>  point at 0.0 0.0
  \setplotarea x from 0 to 1, y from 0.0 to 0.1
\relax
\scriptsize
  \axis bottom ticks short numbered from 0 to 1 by 0.2 /
  \axis left ticks short numbered from 0 to 0.09 by 0.01 /
\footnotesize
\put {$\ell(x)$} [l] at 0.02 0.1
\linethickness=0.6pt
\putrule from 1 0.0 to 1 0.1
\setdashes <3pt>
\putrule from 0 0.02 to 0.7 0.02
\put {$\ell_{init}$} at 0.75 0.02
\putrule from 0.8 0.02 to 1 0.02
\setquadratic
\setsolid
\Black{\relax
\plot    
  0.0000   0.0900
  0.0200   0.0899
  0.0400   0.0897
  0.0600   0.0893
  0.0800   0.0888
  0.1000   0.0881
  0.1200   0.0873
  0.1400   0.0864
  0.1600   0.0854
  0.1800   0.0843
  0.2000   0.0831
  0.2200   0.0819
  0.2400   0.0806
  0.2600   0.0794
  0.2800   0.0781
  0.3000   0.0769
  0.3200   0.0757
  0.3400   0.0746
  0.3600   0.0736
  0.3800   0.0727
  0.4000   0.0719
  0.4200   0.0712
  0.4400   0.0707
  0.4600   0.0703
  0.4800   0.0701
  0.5000   0.0700
  0.5200   0.0701
  0.5400   0.0703
  0.5600   0.0707
  0.5800   0.0712
  0.6000   0.0719
  0.6200   0.0727
  0.6400   0.0736
  0.6600   0.0746
  0.6800   0.0757
  0.7000   0.0769
  0.7200   0.0781
  0.7400   0.0794
  0.7600   0.0806
  0.7800   0.0819
  0.8000   0.0831
  0.8200   0.0843
  0.8400   0.0854
  0.8600   0.0864
  0.8800   0.0873
  0.9000   0.0881
  0.9200   0.0888
  0.9400   0.0893
  0.9600   0.0897
  0.9800   0.0899
  1.0000   0.0900
 /\relax}\relax
\linethickness=0.4pt
\setdashes <3pt>
\setsolid
\Green{\relax   
\plot
  0.0000   0.0845
  0.0200   0.0843
  0.0400   0.0841
  0.0600   0.0839
  0.0800   0.0835
  0.1000   0.0830
  0.1200   0.0824
  0.1400   0.0818
  0.1600   0.0811
  0.1800   0.0803
  0.2000   0.0795
  0.2200   0.0786
  0.2400   0.0777
  0.2600   0.0768
  0.2800   0.0758
  0.3000   0.0749
  0.3200   0.0740
  0.3400   0.0732
  0.3600   0.0724
  0.3800   0.0717
  0.4000   0.0711
  0.4200   0.0706
  0.4400   0.0701
  0.4600   0.0698
  0.4800   0.0696
  0.5000   0.0696
  0.5200   0.0696
  0.5400   0.0698
  0.5600   0.0701
  0.5800   0.0706
  0.6000   0.0711
  0.6200   0.0717
  0.6400   0.0725
  0.6600   0.0732
  0.6800   0.0741
  0.7000   0.0750
  0.7200   0.0759
  0.7400   0.0769
  0.7600   0.0778
  0.7800   0.0787
  0.8000   0.0796
  0.8200   0.0804
  0.8400   0.0811
  0.8600   0.0818
  0.8800   0.0824
  0.9000   0.0830
  0.9200   0.0834
  0.9400   0.0838
  0.9600   0.0840
  0.9800   0.0842
  1.0000   0.0843
/\relax}\relax
\setsolid
\Cyan{\relax   
\plot
  0.0000   0.0845
  0.0200   0.0899
  0.0400   0.0897
  0.0600   0.0893
  0.0800   0.0887
  0.1000   0.0880
  0.1200   0.0872
  0.1400   0.0863
  0.1600   0.0853
  0.1800   0.0842
  0.2000   0.0830
  0.2200   0.0818
  0.2400   0.0806
  0.2600   0.0793
  0.2800   0.0781
  0.3000   0.0769
  0.3200   0.0758
  0.3400   0.0747
  0.3600   0.0737
  0.3800   0.0728
  0.4000   0.0720
  0.4200   0.0713
  0.4400   0.0708
  0.4600   0.0704
  0.4800   0.0702
  0.5000   0.0701
  0.5200   0.0702
  0.5400   0.0704
  0.5600   0.0708
  0.5800   0.0713
  0.6000   0.0720
  0.6200   0.0728
  0.6400   0.0737
  0.6600   0.0747
  0.6800   0.0759
  0.7000   0.0770
  0.7200   0.0783
  0.7400   0.0795
  0.7600   0.0808
  0.7800   0.0820
  0.8000   0.0832
  0.8200   0.0843
  0.8400   0.0853
  0.8600   0.0863
  0.8800   0.0871
  0.9000   0.0878
  0.9200   0.0884
  0.9400   0.0889
  0.9600   0.0892
  0.9800   0.0894
  1.0000   0.0836
/\relax}\relax
\endpicture
}

Figure~\ref{fig:Cauchy2}
shows reconstructions of $\ell(x)$ at 1 per cent noise.
Here the actual curve was defined by 
\begin{equation}\label{ell_act}
\ell(x) = \ell^\dagger(x):=\olell(0.8+0.1\cos(2\pi x))
\end{equation}
with $\olell=0.1$, and the starting value was far from the actual curve (taken to be at $y=0.2\olell$).

The left panel shows the case of Dirichlet conditions on the interface. 
No further progress in convergence took place after iteration~5.
The lateral boundary conditions were of homogeneous Neumann type in order to avoid singularities near the corners.

\begin{figure}[ht]
\includegraphics[width=\textwidth]{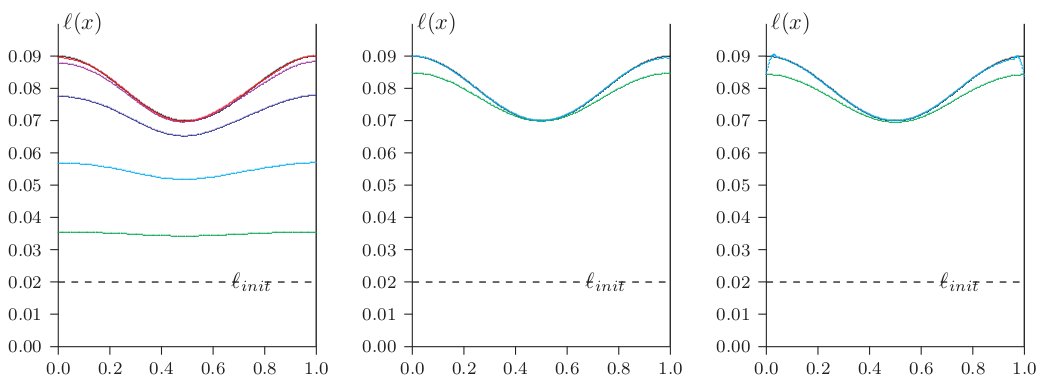}
\caption{\small {\bf Recovery of $\ell(x)$ from data in with $1\%$ noise:
Dirichlet case (left),
Neumann case (middle),
impedance case (right).
}}
\label{fig:Cauchy2}
\end{figure}
In figure~\ref{fig:Cauchy2} the 
middle and rightmost figures we show only the first (green) and
second (blue) iterations as the latter was effective convergence.
In the leftmost figure we additionally show iteration~3 (navy blue),
iteration~4 (purple) and iteration~5 (red); the latter being effective
convergence.

The relative errors at different noise levels are given in the second column of Table~\ref{tab:Cauchy2}. 

\begin{center}
\begin{table}[ht]
\begin{tabular}{|c||c|c|c|c|}
\hline
 &\multicolumn{4}{c|}{$\|\ell-\ell^\dagger\|_{L^2(\Omega)}/\|\ell^\dagger\|_{L^2(\Omega)}$}\\
\hline
noise level & (D), $\olell=0.1$ &  (D), $\olell=0.5$ &  (N), $\olell=0.1$ &  (I), $\olell=0.1$ \\
\hline
 1\% &0.0038 &0.0158 &0.0018 &0.0077\\
 2\% &0.0084 &0.0205 &0.0093 &0.0087\\
 5\% &0.0198 &0.0380 &0.0144 &0.0110\\
10\% &0.0394 &0.0735 &0.0563 &0.0158\\
\hline
\end{tabular}
\caption{Relative errors for reconstructions of $\ell$ at several noise levels. \label{tab:Cauchy2}}
\end{table}
\end{center}

The same runs were also done with a larger distance $\olell$ and the actual curve still according to \eqref{ell_act}. This resulted in the relative errors shown in column three of Table~\ref{tab:Cauchy2}.

\smallskip

The middle and right panels of Figure~\ref{fig:Cauchy2}
show  reconstructions of $\ell(x)$ in the case
of Neumann and impedance conditions ($\gamma=0.1$) on the interface.
The starting value was again relatively far from the actual curve
and there was again $1\%$ noise in the data.
No further progress in convergence took place after iteration~2.
For the relative errors at different noise levels, see the last two columns of Table~\ref{tab:Cauchy2}.

\section{Problem 3}\label{sec:Cauchy3}
\begin{problem}\label{Cauchy3}
Given two pairs of Cauchy data $(f_1,g_1)$, $(f_2,g_2)$ in
\[
\begin{aligned}
-\triangle u_j =&0  &&\mbox{ in }x\in(0,L)\,, \ y\in(0,\ell(x))\\
B_0u_j(0,y)=&0 &&\ x\in(0,\ell(0))\\ 
B_Lu_j(L,y)=&0 &&\ x\in(0,\ell(L))\\
u_j(x,0)=f_j(x)\,, \ u_{jy}(x,0)=&g_j(x)&& \ x\in\Omega
\end{aligned}
\]
for $j=1,2$,
find $\ell:(0,L)\to(0,\bar{l})$ and $\gamtil:(0,L)\to(0,\infty)$ such that 
\begin{equation}\label{eqn:Bellgamtil}
\begin{aligned}
0=B_{\ell,\gamtil}u_j:=&\partial_\nnu u_j(x,\ell(x)) + \gamtil(x) u_j(x,\ell(x))\\
=&u_{jy}(x,\ell(x))-\ell'(x) u_{jx}(x,\ell(x)) + \gamtil(x) u_j(x,\ell(x))
\quad x\in(0,L)
\end{aligned}
\end{equation}
where $\gamtil$ is the combined coefficient defined by $\gamtil(x)=\sqrt{1+\ell'(x)^2}\gamma(x)$.
(Note that $\nnu$ is the non-normalised outer normal direction, but for the zero level set this does not matter.)
\end{problem}

The setting is as in Problem~\ref{Cauchy2} otherwise and using the notation \eqref{Dofell}
we rewrite the forward problem as
\begin{equation}\label{eqn:forwardCauchy3}
\begin{aligned}
-\triangle u_j =&0 && \ \mbox{ in }D(\ell)\\
u_j=&f_j&&\ \mbox{ on }\Gamma_1\\
Bu_j=&0&&\ \mbox{ on }\Gamma_2(\ell)\\
B_{\ell,\gamtil}u_j
=&0&& \ \mbox{ on }\Gamma_0(\ell).
\end{aligned}
\end{equation}
Note that $u_j$ actually satisfies the Poisson equation on the hold-all domain $D(\olell)$ with a fixed upper boundary defined by $\olell\geq\ell$.
We also point out that using the weak form of the forward problem
\[
\begin{aligned}
&u-f_j\in H_\diamondsuit^1(D(\ell)):=\{w\in H^1(\Omega)\,: \, Bw=0\}\text{ on }\Gamma_2(\ell)\,, \ w=0\text{ on }\Gamma_2\} \\
&\text{ and for all }w\in H_\diamondsuit^1(D(\ell))\, : \
\int_0^L\Bigl(\int_0^{\ell(x)}(\nabla u\cdot\nabla w)(x,y)\, dy +  \gamtil(x) (u\cdot w)(x,\ell(x))\Bigr)\, dx = 0
\end{aligned}
\]
no derivative of $\ell$ nor $\gamtil$ is needed for computing $u$.

The forward operator $F=(F_1,F_2)$ is defined by
\[
F_j:\ell\mapsto u_{jy}(x,0)-g_j(x)\mbox{ where $u_j$ solves \eqref{eqn:forwardCauchy3}, }j\in\{1,2\}. 
\]
Its linearisation is defined by $F_j'(\ell)[\dl,\dgam]=v_{j\,y}\vert_{\Gamma_1}$, where $v_j$ solves
\[
\begin{aligned}
-\triangle v_j =&0&&\ \mbox{ in }D(\ell)\\
v_j=&0&&\ \mbox{ on }\Gamma_1\\
Bv_j=&0&&\ \mbox{ on }\Gamma_2(\ell)\\
B_{\ell,\gamtil}v_j
=&G(u_j,\ell,\gamtil)(\dl,\dgam) &&\ \mbox{ on }\Gamma_0(\ell)
\end{aligned}
\]
cf., \eqref{eqn:Bellgamtil}, where 
\begin{equation}\label{eqn:G}
\begin{aligned}
G(u_j,\ell,\gamtil)(\dl,\dgam))(x)
:=& -\bigl(u_{j\,yy}(x,\ell(x))-\ell'(x)u_{j\,xy}(x,\ell(x))+\gamtil(x) u_{j\,y}(x,\ell(x))\bigr)\dl\\
&+\dl'(x)u_{j\,x}(x,\ell(x))
-\dgam(x)u_j(x,\ell(x)).
\end{aligned}
\end{equation}
Using the PDE we have the identity
\[
G(u_j,\ell,\gamtil)(\dl,\dgam)
=\tfrac{d}{dx}\Bigl[\dl(x)\,u_{jx}(x,\ell(x))\Bigr]
-\dl(x)\gamtil(x)u_{jy}(x,\ell(x))-\dgam(x) u_j(x,\ell(x)).
\]

Thus, computation of a Newton step $(\dl,\dgam)=(\dl^{(k)},\dgam^{(k)})$ starting from some iterate $(\ell^{(k)},\gamtil^{(k)})=(\ell,\gamtil)$ amounts to solving the system
\begin{equation}\label{eqn:Newtonstep_Cauchy3}
\left.\begin{array}{rll}
-\triangle z_j =&0&  \mbox{ in }D(\ell)\\
z_j=&f_j&\mbox{ on }\Gamma_1\\
Bz_j=&0&\mbox{ on }\Gamma_2(\ell)\\
z_{j\,y}-g_j=&0&\mbox{ on }\Gamma_1\\
B_{\ell,\gamtil}z_j
=&G(u_j,\ell,\gamtil)(\dl,\dgam)&\mbox{ on }\Gamma_0(\ell)
\end{array}\right\}\quad j\in\{1,2\}
\end{equation}
(note that $z_j=u_j+v_j$ and $z_{j\,y}\vert_{\Gamma_1}-g_j=F_j(\ell)+F_j'(\ell)(\dl,\dgam)$). \hfill\break

With pre-computed (regularised) solutions $\bar{z}_j$ of the Cauchy problem on the rectangular hold-all domain
\begin{equation}\label{eqn:zbarj}
\begin{aligned}
-\triangle \bar{z}_j =&0  \mbox{ in }D(\olell)=(0,L)\times(0,\olell)\\
\bar{z}_j=&f_j\mbox{ on }\Gamma_1\\
B\bar{z}_j=&0\mbox{ on }\Gamma_2(\olell)=\{0,L\}\times(0,\olell)\\
\bar{z}_{j\,y}-g_j=&0\mbox{ on }\Gamma_1\\
\end{aligned}
\end{equation}
this reduces to resolving the following system for $(\dl,\dgam)$ on $\Gamma_0(\ell)$
\begin{equation}\label{eqn:sysNewtonCauchy3}
\begin{aligned}
G(u_1,\ell,\gamtil)(\dl,\dgam)=& B_{\ell,\gamtil}\bar{z}_1\\
G(u_2,\ell,\gamtil)(\dl,\dgam)=& B_{\ell,\gamtil}\bar{z}_2.
\end{aligned}
\end{equation}

To obtain more explicit expressions for $\dl$ and $\dgam$ from \eqref{eqn:sysNewtonCauchy3}, one can apply an elimination strategy, that is, multiply the boundary condition on $\Gamma_0(\ell)$ with $u_{j\pm1}(x)$ 
and subtract, to obtain 
\[
\begin{aligned}
B_{\ell,\gamtil}v_1 \, u_2 - B_{\ell,\gamtil}v_2 \, u_1 
&=u_2\tfrac{d}{dx}\Bigl[\dl\,u_{1x}\Bigr]
-u_1\tfrac{d}{dx}\Bigl[\dl\,u_{2x}\Bigr]-\dl\gamtil\Bigl(u_{1y}u_2-u_{2y}u_1\Bigr)\\
&= \tfrac{d}{dx}\Bigl[\dl\Bigl(u_{1x}u_2-u_{2x}u_1\Bigr)\Bigr]
-\dl\Bigl(\ell'\bigl(u_{1x}u_{2y}-u_{2x}u_{1y}\bigr)
+\gamtil\bigl(u_{1y}u_2-u_{2y}u_1\bigr)
\Bigr)\\
&=:\tfrac{d}{dx}\Bigl[\dl\altil\Bigr]-\dl\betil
\end{aligned}
\]
where we have skipped the arguments $(x)$ of $b_1$, $b_2$, $\dl$, $\dgam$, $\ell'$ and $(x,\ell(x))$ of $u_1$, $u_2$ and its derivatives for better readability. 

With the pre-computed (regularised) solutions $\bar{z}_j$ of the Cauchy problem \eqref{eqn:zbarj}
one can therefore carry out a Newton step by computing $u_j^{(k)}=u_j$ from the well-posed mixed elliptic boundary value problem (*) with $\ell=\ell^{(k)}$, $\gamtil=\gamtil^{(k)}$ and updating
\begin{equation}\label{eqn:dl}
\begin{aligned}
&\ell^{(k+1)}(x)=\ell^{(k)}(x)-\dl(x)\\
\mbox{where\quad}\dl(x)&=\tfrac{1}{\altil[\ell^{(k)},u^{(k)}](x)}
\Bigl\{
\exp\Bigl(-\int_0^x\tfrac{\betil[\ell^{(k)},u^{(k)}]}{\altil[\ell^{(k)},u^{(k)}]}(s)\, ds\Bigr) \dl(0)\altil[\ell^{(k)},u^{(k)}](0)\\
&\hspace*{2cm}+\int_0^x \tilde{b}(s) \exp\Bigl(-\int_s^x\tfrac{\betil[\ell^{(k)},u^{(k)}]}{\altil[\ell^{(k)},u^{(k)}]}(t)\, dt\Bigr)\, ds
\Bigr\}\\
\mbox{with \quad}
\tilde{b}(x)&=\bigl(\partial_\nnu \bar{z}_1(x,\ell^{(k)}(x))+\gamtil^{(k)}(x)\bar{z}_1(x,\ell^{(k)}(x))\bigr)u_2^{(k)}(x,\ell^{(k)}(x))\\
&\hspace*{2cm}-\bigl(\partial_\nnu \bar{z}_2(x,\ell^{(k)}(x))+\gamtil^{(k)}(x)\bar{z}_2(x,\ell^{(k)}(x))\bigr)u_1^{(k)}(x,\ell^{(k)}(x))
\end{aligned}
\end{equation}
\begin{equation}\label{eqn:dgam}
\begin{aligned}
\gamtil^{(k+1)}(x)&=\gamtil^{(k)}(x)-\dgam(x)\\
\mbox{where\quad}\dgam(x)&=\tfrac{1}{u_1^{(k)}(x,\ell^{(k)}(x))}\Bigl(\tfrac{d}{dx}\Bigl[\dl(x)\,u_{1x}^{(k)}(x,\ell^{(k)}(x))\Bigr]
-\dl(x)\gamtil^{(k)}(x)u_{1y}^{(k)}(x,\ell(x))-b_1(x)\Bigr)\\
\hspace*{2.2cm}&=\tfrac{1}{u_1^{(k)}(x,\ell^{(k)}(x))}\Bigl(\tfrac{d}{dx}\Bigl[\dl(x)\altil[\ell^{(k)},u^{(k)}](x)\Bigr]\\
&\hspace*{1cm}
-\dl(x)\Bigl(\gamtil^{(k)}(x)u_{1y}^{(k)}(x,\ell(x))+\tfrac{d}{dx}\Bigl[\tfrac{u_{1x}^{(k)}(x,\ell^{(k)}(x))}{\altil[\ell^{(k)},u^{(k)}](x)}\Bigr]\Bigr)-b_1(x)\Bigr)
\end{aligned}
\end{equation}
Since this elimination procedure needs second derivative computations,
we do not use it for the  reconstruction but stay with a simultaneous
computation of the Newton step $(\dl,\dgam)$ from \eqref{eqn:sysNewtonCauchy3}.
Still, \eqref{eqn:dl}, \eqref{eqn:dgam} will be useful in the proof of
linearised uniqueness to follow.

\subsection{Reconstructions}\label{sec:reconCauchy3}

\input colordvi

\input pictex
\font\smallsymbol = cmmi8
\newdimen\xfiglen \newdimen\yfiglen
\xfiglen=2.7 true in
\yfiglen=2.5 true in
%
%
\setbox\figurefour=\vbox{\hsize=0.6true in
\beginpicture
  \setcoordinatesystem units <0.06\xfiglen,0.05\yfiglen> 
  \setplotarea x from 0 to 2, y from 0 to 2
\linethickness=0.5pt
\scriptsize
   \setsolid
   \Black{\relax \putrule from 0 2 to 1 2 \relax}\relax
   \put {actual}  [l] at 1.2 2
   \Red{\relax \putrule from 0 1 to 1 1 \relax}\relax
   \put {reconstruction}  [l] at 1.2 1
\endpicture
}
\setbox\figurefive=\vbox{\hsize=\xfiglen   
\beginpicture
  \setcoordinatesystem units <\xfiglen,12.0\yfiglen>  point at 0.0 0.05
  \setplotarea x from 0 to 1, y from 0.05 to 0.1
\scriptsize
  \axis bottom ticks short numbered from 0 to 1 by 0.2 /
  \axis left ticks short numbered from 0.05 to 0.1 by 0.01 /
\put {0.05} [r] at -0.02 0.05 \putrule from -0.01 0.05 to 0 0.05
\put {0.10} [r] at -0.02 0.1  \putrule from -0.01 0.1 to 0 0.1
\footnotesize
\put {$\ell(x)$} [l] at 0.02 0.1
\linethickness=0.6pt
\setsolid
\put {\copy\figurefour} [rb] at 1.0 0.055
\setquadratic
\Black{
\plot    
  0.0000   0.0900
  0.0200   0.0899
  0.0400   0.0897
  0.0600   0.0893
  0.0800   0.0888
  0.1000   0.0881
  0.1200   0.0873
  0.1400   0.0864
  0.1600   0.0854
  0.1800   0.0843
  0.2000   0.0831
  0.2200   0.0819
  0.2400   0.0806
  0.2600   0.0794
  0.2800   0.0781
  0.3000   0.0769
  0.3200   0.0757
  0.3400   0.0746
  0.3600   0.0736
  0.3800   0.0727
  0.4000   0.0719
  0.4200   0.0712
  0.4400   0.0707
  0.4600   0.0703
  0.4800   0.0701
  0.5000   0.0700
  0.5200   0.0701
  0.5400   0.0703
  0.5600   0.0707
  0.5800   0.0712
  0.6000   0.0719
  0.6200   0.0727
  0.6400   0.0736
  0.6600   0.0746
  0.6800   0.0757
  0.7000   0.0769
  0.7200   0.0781
  0.7400   0.0794
  0.7600   0.0806
  0.7800   0.0819
  0.8000   0.0831
  0.8200   0.0843
  0.8400   0.0854
  0.8600   0.0864
  0.8800   0.0873
  0.9000   0.0881
  0.9200   0.0888
  0.9400   0.0893
  0.9600   0.0897
  0.9800   0.0899
  1.0000   0.0900
 /\relax}\relax
%
\setsolid
\Red{\relax   
\plot
  0.0000   0.0900
  0.0200   0.0899
  0.0400   0.0898
  0.0600   0.0897
  0.0800   0.0896
  0.1000   0.0894
  0.1200   0.0891
  0.1400   0.0886
  0.1600   0.0879
  0.1800   0.0870
  0.2000   0.0859
  0.2200   0.0845
  0.2400   0.0830
  0.2600   0.0814
  0.2800   0.0797
  0.3000   0.0781
  0.3200   0.0765
  0.3400   0.0750
  0.3600   0.0737
  0.3800   0.0725
  0.4000   0.0716
  0.4200   0.0708
  0.4400   0.0703
  0.4600   0.0700
  0.4800   0.0698
  0.5000   0.0699
  0.5200   0.0701
  0.5400   0.0705
  0.5600   0.0711
  0.5800   0.0718
  0.6000   0.0726
  0.6200   0.0735
  0.6400   0.0744
  0.6600   0.0754
  0.6800   0.0763
  0.7000   0.0773
  0.7200   0.0784
  0.7400   0.0795
  0.7600   0.0807
  0.7800   0.0821
  0.8000   0.0836
  0.8200   0.0850
  0.8400   0.0864
  0.8600   0.0876
  0.8800   0.0885
  0.9000   0.0893
  0.9200   0.0898
  0.9400   0.0901
  0.9600   0.0903
  0.9800   0.0903
  1.0000   0.0900
 /\relax}\relax
\setsolid
\setdashes <3pt>
\Black{
\plot    
  0.0000   0.0900
  0.0200   0.0900
  0.0400   0.0900
  0.0600   0.0900
  0.0800   0.0900
  0.1000   0.0900
  0.1200   0.0900
  0.1400   0.0900
  0.1600   0.0900
  0.1800   0.0900
  0.2000   0.0900
  0.2200   0.0900
  0.2400   0.0900
  0.2600   0.0900
  0.2800   0.0900
  0.3000   0.0900
  0.3200   0.0900
  0.3400   0.0900
  0.3600   0.0900
  0.3800   0.0900
  0.4000   0.0900
  0.4200   0.0900
  0.4400   0.0900
  0.4600   0.0900
  0.4800   0.0900
  0.5000   0.0900
  0.5200   0.0900
  0.5400   0.0900
  0.5600   0.0900
  0.5800   0.0900
  0.6000   0.0900
  0.6200   0.0900
  0.6400   0.0900
  0.6600   0.0900
  0.6800   0.0900
  0.7000   0.0900
  0.7200   0.0900
  0.7400   0.0900
  0.7600   0.0900
  0.7800   0.0900
  0.8000   0.0900
  0.8200   0.0900
  0.8400   0.0900
  0.8600   0.0900
  0.8800   0.0900
  0.9000   0.0900
  0.9200   0.0900
  0.9400   0.0900
  0.9600   0.0900
  0.9800   0.0900
  1.0000   0.0900
 /\relax}\relax
\endpicture
}
\setbox\figuresix=\vbox{\hsize=\xfiglen  
\beginpicture
  \setcoordinatesystem units <\xfiglen,\yfiglen>  point at 0.0 0.8
  \setplotarea x from 0 to 1, y from 0.8 to 1.4
\relax
\scriptsize
  \axis bottom ticks short numbered from 0 to 1 by 0.2 /
  \axis left ticks short numbered from 0.8 to 1.4 by 0.2 /
\footnotesize
\put {$\gamma(x)$} [l] at 0.02 1.4
\linethickness=0.6pt
\setsolid
\setquadratic
\Black{
\plot    
  0.0000   1.0000
  0.0200   1.0000
  0.0400   1.0001
  0.0600   1.0003
  0.0800   1.0005
  0.1000   1.0007
  0.1200   1.0011
  0.1400   1.0039
  0.1600   1.0142
  0.1800   1.0375
  0.2000   1.0769
  0.2200   1.1307
  0.2400   1.1914
  0.2600   1.2479
  0.2800   1.2879
  0.3000   1.3023
  0.3200   1.2876
  0.3400   1.2472
  0.3600   1.1905
  0.3800   1.1296
  0.4000   1.0757
  0.4200   1.0363
  0.4400   1.0130
  0.4600   1.0029
  0.4800   1.0002
  0.5000   1.0000
  0.5200   1.0000
  0.5400   1.0001
  0.5600   1.0003
  0.5800   1.0005
  0.6000   1.0007
  0.6200   1.0009
  0.6400   1.0012
  0.6600   1.0014
  0.6800   1.0016
  0.7000   1.0018
  0.7200   1.0019
  0.7400   1.0020
  0.7600   1.0020
  0.7800   1.0019
  0.8000   1.0018
  0.8200   1.0016
  0.8400   1.0014
  0.8600   1.0012
  0.8800   1.0009
  0.9000   1.0007
  0.9200   1.0005
  0.9400   1.0003
  0.9600   1.0001
  0.9800   1.0000
  1.0000   1.0000
 /\relax}\relax
\Red{\relax
\plot    
  0.0000   1.0116
  0.0200   1.0109
  0.0400   1.0085
  0.0600   1.0054
  0.0800   1.0034
  0.1000   1.0046
  0.1200   1.0112
  0.1400   1.0250
  0.1600   1.0466
  0.1800   1.0757
  0.2000   1.1104
  0.2200   1.1480
  0.2400   1.1848
  0.2600   1.2168
  0.2800   1.2405
  0.3000   1.2528
  0.3200   1.2521
  0.3400   1.2378
  0.3600   1.2113
  0.3800   1.1750
  0.4000   1.1326
  0.4200   1.0883
  0.4400   1.0462
  0.4600   1.0102
  0.4800   0.9832
  0.5000   0.9665
  0.5200   0.9604
  0.5400   0.9634
  0.5600   0.9733
  0.5800   0.9868
  0.6000   1.0005
  0.6200   1.0115
  0.6400   1.0178
  0.6600   1.0187
  0.6800   1.0148
  0.7000   1.0079
  0.7200   1.0001
  0.7400   0.9931
  0.7600   0.9877
  0.7800   0.9835
  0.8000   0.9804
  0.8200   0.9783
  0.8400   0.9777
  0.8600   0.9795
  0.8800   0.9840
  0.9000   0.9913
  0.9200   1.0005
  0.9400   1.0105
  0.9600   1.0203
  0.9800   1.0311
  1.0000   1.0794
 /\relax}\relax
\setdashes <3pt>
\Black{
\plot    
  0.0000   1.0000
  0.0200   1.0000
  0.0400   1.0000
  0.0600   1.0000
  0.0800   1.0000
  0.1000   1.0000
  0.1200   1.0000
  0.1400   1.0000
  0.1600   1.0000
  0.1800   1.0000
  0.2000   1.0000
  0.2200   1.0000
  0.2400   1.0000
  0.2600   1.0000
  0.2800   1.0000
  0.3000   1.0000
  0.3200   1.0000
  0.3400   1.0000
  0.3600   1.0000
  0.3800   1.0000
  0.4000   1.0000
  0.4200   1.0000
  0.4400   1.0000
  0.4600   1.0000
  0.4800   1.0000
  0.5000   1.0000
  0.5200   1.0000
  0.5400   1.0000
  0.5600   1.0000
  0.5800   1.0000
  0.6000   1.0000
  0.6200   1.0000
  0.6400   1.0000
  0.6600   1.0000
  0.6800   1.0000
  0.7000   1.0000
  0.7200   1.0000
  0.7400   1.0000
  0.7600   1.0000
  0.7800   1.0000
  0.8000   1.0000
  0.8200   1.0000
  0.8400   1.0000
  0.8600   1.0000
  0.8800   1.0000
  0.9000   1.0000
  0.9200   1.0000
  0.9400   1.0000
  0.9600   1.0000
  0.9800   1.0000
  1.0000   1.0000
 /\relax}\relax
\endpicture
}

In Figure~\ref{fig:ell_and_gamma}
we show a simultaneous reconstruction of $\ell(x)$ and $\gamtil(x)$
obtained from excitations with $f_1(x)=1 + x + x^2$ and $f_2(x)=4x^2 - 3x^3$
(chosen to comply with the impedance conditions $Bu=\partial_\nu u+u=0$ at $x\in\{0,L\}$, $L=1$).

\begin{figure}[ht]
\includegraphics[width=\textwidth]{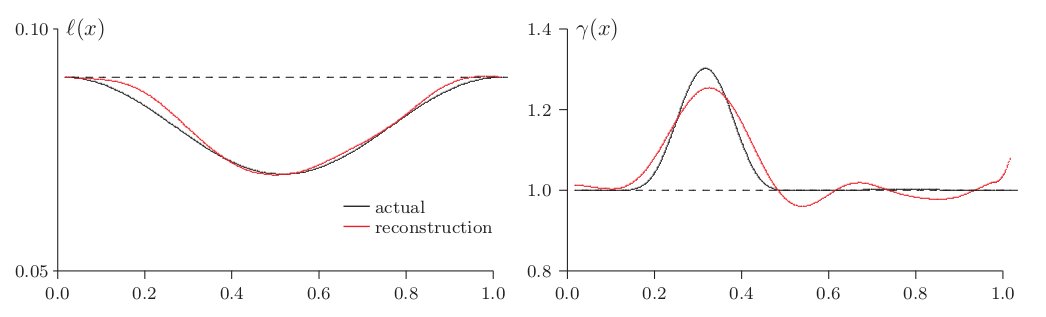}
\caption{\small {\bf Simultaneous recovery of $\ell(x)$ and $\gamtil(x)$}}
\label{fig:ell_and_gamma}
\end{figure}

Assuming the boundary values of $\ell$ and $\gamtil$ to be known, we used them in our (linear; in the concrete setting constant) starting guesses.
The actual values are shown in black (solid), the reconstructed in red, the starting guesses as dashed lines.

The relative errors at different noise levels are listed in Table~\ref{tab:Cauchy3}.
\begin{center}
\begin{table}[ht]
\begin{tabular}{|c||c|c|}
\hline
noise level &$\|\ell-\ell^\dagger\|_{L^2(\Omega)}/\|\ell^\dagger\|_{L^2(\Omega)}$&
$\|\gamtil-\gamtil^\dagger\|_{L^2(\Omega)}/\|\gamtil^\dagger\|_{L^2(\Omega)}$\\
\hline
 1\% &0.0145 &0.0251\\
 2\% &0.0152 &0.0263\\
 5\% &0.0191 &0.0355\\
10\% &0.0284 &0.0587\\
\hline
\end{tabular}
\caption{Relative errors for reconstructions of $\ell$ and $\gamtil$ at several noise levels. \label{tab:Cauchy3}}
\end{table}
\end{center}

While the numbers suggest approximately logarithmic convergence 
and this is in line with what is to be expected for such a severely ill-posed
problem, our result  in Section~\ref{sec:convCauchy3} is about convergence only.
Proving rates would go beyond the scope of the objective of this paper.

\subsection{Convergence}\label{sec:convCauchy3}
Proving convergence of iterative regularisation methods for nonlinear ill-posed problems always requires certain conditions on the nonlinearity of the forward operator.
These can be verified for a slightly modified formulation of the problem; in particular, instead of a reduced formulation involving the parameter-to-state map as used in Section~\ref{sec:reconCauchy3}, we consider the all-at-once formulation  (involving a second copy of $\gamtil$)
\begin{equation}\label{FtilP}
\begin{aligned}
\Ftil_1(u_1,\ell,\gamtil_1)=&0 \\ 
\Ftil_2(u_2,\ell,\gamtil_2)=&0 \\ 
P(\gamtil_1,\gamtil_2)=&0,
\end{aligned}
\end{equation}
where 
\begin{equation}\label{eqn:Ftilj}
\Ftil_j:(u_j,\ell,\gamtil_j)\mapsto
\left(\begin{array}{c}
-\triangle u_j \\
u_j\vert_{\Gamma_1}-f_j\\
u_{j\,y}\vert_{\Gamma_1}-g_j\\
Bu_j\vert_{\Gamma_2(\ell)}\\
B_{\ell,\gamtil}u_j\vert_{\Gamma_0(\ell)}
\end{array}\right)\,,\ j\in\{1,2\}, \quad
P(\gamtil_1,\gamtil_2):=(\gamtil_1-\gamtil_2,\ell(0)-\ell^0)
\end{equation}
with a given endpoint value $\ell^0$.
(The first component of $\Ftil_j$ means that we consider the Poisson equation on the hold-all domain $D(\olell)$.)

\subsubsection*{Range invariance of the forward operator}
The forward operator $\Ftil:\mathbb{X}\to\mathbb{Y}$ defined by $\Ftil(u_1,u_2,\ell,\gamtil_1,\gamtil_2)=(\Ftil_1(u_1,\ell,\gamtil_1),\Ftil_2(u_2,\ell,\gamtil_2)$ along with its linearisation
\begin{equation}\label{eqn:Ftilprime}
\Ftil_j'(u_j,\ell,\gamtil_j)(\du_j,\dl,\dgam_j) =
\left(\begin{array}{c}
-\triangle \du_j \\
\du_j\vert_{\Gamma_1}\\
\du_{j\,y}\vert_{\Gamma_1}\\
B\du_j\vert_{\Gamma_2(\ell)}\\
B_{\ell,\gamtil_j}\du_j\vert_{\Gamma_0(\ell)}-G(u_j,\ell,\gamtil)(\dl,\dgam)
\end{array}\right)
\end{equation}
and the operator $r=(r_{u\,1},r_{u\,2},r_\ell,r_{\gamtil\,1},r_{\gamtil\,2})$
defined so that 
$[B_{\ell,\gamtil_j}-B_{\ell_0,\gamtil_{0,j}}]u_j
+G(u_{0,j},\ell_0,\gamtil_{0,j})(\ell-\ell_0),r_{\gamma\,j}(u_1,u_2,\ell,\gamtil_1,\gamtil_2))=0$, that is, 
\[
\begin{aligned}
&r_{u\,j}(u_1,u_2,\ell,\gamtil_1,\gamtil_2)=u_j-u_{0,j}
\\
&r_\ell(u_1,u_2,\ell,\gamtil_1,\gamtil_2)=\ell-\ell_0\\
&r_{\gamma\,j}(u_1,u_2,\ell,\gamtil_1,\gamtil_2)=
\frac{1}{u_{0,j}(\ell_0)}\\
&\quad\times\Bigl(
-\bigl(u_{0,j\,yy}(\ell_0)-\ell_0'(x)u_{0,j\,xy}(\ell_0)+\gamtil_0(x) u_{0,j\,y}(\ell_0)\bigr)(\ell-\ell_0)+(\ell'-\ell_0')u_{0,j\,x}(\ell_0)\\
&\qquad+ \bigl(u_{j\,y}(\ell)-\ell'(x)u_{j\,x}(\ell)+\gamtil(x) u_{j}(\ell)\bigr)
-\bigl(u_{j\,y}(\ell_0)-\ell_0'(x)u_{j\,x}(\ell_0)+\gamtil_0(x) u_{j}(\ell_0)\bigr)
\Bigr)
\end{aligned}
\]
$j\in\{1,2\}$, satisfies the differential range invariance condition
\begin{equation}\label{rangeinvar_diff}
\textup{for all } \xi\in U \, \exists r(\xi)\in \mathbb{X}:=V^2\times X^\ell\times (X^\gamtil)^2\,: \
\Ftil(\xi)-\Ftil(\xi_0)= \Ftil'(\xi_0)\,r(\xi),
\end{equation}
in a neighborhood $U$ of the reference point $\xi_0:=(u_{0,1},u_{0,2},\ell_0,\gamtil_{0,1},\gamtil_{0,2})$. 
Here we use the abbreviation $u(\ell)$ for $(u(\ell))(x)=u(x,\ell(x))$ and 
\[
\xi := (u_1,u_2,\ell,\gamtil_1,\gamtil_2).
\]
The difference between $r$ and the shifted identity can be written as 
\[
\begin{aligned}
&r(u_1,u_2,\ell,\gamtil_1,\gamtil_2)-\bigl((u_1,u_2,\ell,\gamtil_1,\gamtil_2)
-(u_{0,1},u_{0,2},\ell_0,\gamtil_{0,1},\gamtil_{0,2})\bigr)\\
&=\bigl(0,0,0,\tfrac{1}{u_{0,1}(\ell_0)}(I_1+II_1+III_1),\tfrac{1}{u_{0,2}(\ell_0)}(I_2+II_2+III_2)\bigr)^T, 
\end{aligned}
\]
with 
\[
\begin{aligned}
I_j=&\gamtil_j(u_j(\ell)-u_{0,j}(\ell_0)-\gamtil_{0,j}u_{0,j\,y}(\ell_0)(\ell-\ell_0)\\
=& \gamtil_j\Bigl((u_j-u_{0,j})(\ell)-(u_j-u_{0,j})(\ell_0)\Bigr)
+(\gamtil_j-\gamtil_{0,j})(u_{0,j}(\ell)-u_{0,j}(\ell_0))\\
&\qquad+\gamtil_{0,j}\Bigl(u_{0,j}(\ell)-u_{0,j}(\ell_0)-u_{0,j\,y}(\ell_0)(\ell-\ell_0)\Bigr)
\\
II_j=&u_{0,j\,y}(\ell)-u_{0,j\,y}(\ell_0)-u_{0,j\,yy}(\ell_0)(\ell-\ell_0)
\\
III_j=&-\ell_{0,j}'\Bigl(u_{0,j\,x}(\ell)-u_{0,j\,x}(\ell_0)-u_{0,j\,xy}(\ell_0)(\ell-\ell_0)\Bigr)\\
&\qquad-\ell'\Bigl((u_{j\,x}-u_{0,j\,x})(\ell)-(u_{j\,x}-u_{0,j\,x})(\ell_0)\Bigr)
\\
&\qquad-(\ell'-\ell_0')\Bigl(u_{0,j\,x}(\ell)-u_{0,j\,x}(\ell_0)\, + \, 
u_{j\,x}(\ell_0)-u_{0,j\,x}(\ell_0)\Bigr).
\end{aligned}
\]
The function spaces are supposed to be chosen according to   
\[
V \subseteq W^{2,\infty}(D(\olell)), \quad 
X^\ell= W^{1,p}(0,L), \quad 
X^\gamtil= L^p(0,L), \quad \text{ with }p>1 
\]
so that $W^{1,p}(0,L)$ continuously embeds into $L^\infty(0,L)$, 
\footnote{which also allows to guarantee $\ell\leq\olell$ for all $\ell$ with $\|\ell-\ell_0\|_{X^\ell}$ small enough} 
e.g. 
\begin{equation}\label{eqn:spacesHilbert}
X^\ell= H^1(0,L), \quad 
X^\gamtil= L^2(0,L),
\end{equation}
in order to work with Hilbert spaces.
In this setting, using
\[
\begin{aligned}
&|v(x,\ell(x))-v(x,\ell_0(x))|
=\left|\int_0^1v_y(x,\ell_0(x)+s(\ell(x)-\ell_0(x)))\, ds\, (\ell(x)-\ell_0(x))\right|\\
&\qquad\leq \|v_y\|_{L^\infty(D(\bar(\ell)))}\, |\ell(x)-\ell_0(x)|\\[1ex]
&|v(x,\ell(x))-v(x,\ell_0(x))-v_y(x,\ell_0(x))|\\
&\qquad=\left|\int_0^1\int_0^1v_{yy}(x,\ell_0(x)+s's(\ell(x)-\ell_0(x)))\,s\, ds\, ds'\, 
(\ell(x)-\ell_0(x))^2\right|\\
&\qquad \leq \|v_{yy}\|_{L^\infty(D(\bar(\ell)))}\, |\ell(x)-\ell_0(x)|^2
\end{aligned}
\]
we can further estimate 
\[
\begin{aligned}
\|I_j\|_{L^p(0,L)}
\leq&\|\gamtil_j\|_{L^p(0,L)} \|u_{j\,y}-u_{0,j\,y}\|_{L^\infty(D(\olell)}
\|\ell-\ell_0\|_{L^\infty(0,L)}\\
&+\|\gamtil_j-\gamtil_{0,j}\|_{L^p(0,L)} \|u_{0,j\,y}\|_{L^\infty(D(\olell)}
\|\ell-\ell_0\|_{L^\infty(0,L)}\\
&+\tfrac12\|\gamtil_{0,j}\|_{L^p(0,L)} \|u_{0,j\,yy}\|_{L^\infty(D(\olell)}
\|\ell-\ell_0\|_{L^\infty(0,L)}^2
\\
\|II_j\|_{L^p(0,L)}
\leq& L^{1/p}\Bigl(
\tfrac12\|u_{0,j\,yyy}\|_{L^\infty(D(\olell)} \|\ell-\ell_0\|_{L^\infty(0,L)}^2
+\|u_{j\,yy}-u_{0,j\,yy}\|_{L^\infty(D(\olell)} \|\ell-\ell_0\|_{L^\infty(0,L)}
\Bigr)
\\
\|III_j\|_{L^p(0,L)}
\leq&\|\ell_{0,j}'\|_{L^p(0,L)} \|u_{0,j\,xyy}\|_{L^\infty(D(\olell)} \|\ell-\ell_0\|_{L^\infty(0,L)}^2\\
&+\|\ell'\|_{L^p(0,L)} \|u_{j\,xy}-u_{0,j\,xy}\|_{L^\infty(D(\olell)}
\|\ell-\ell_0\|_{L^\infty(0,L)}\\ 
&+\|\ell'-\ell_0'\|_{L^p(0,L)} \Bigl(\|u_{0,j\,xy}\|_{L^\infty(D(\olell)}
\|\ell-\ell_0\|_{L^\infty(0,L)}
+ \|u_{j\,x}(\ell_0)-u_{0,j\,x}(\ell_0)\|_{L^\infty(0,L)}\Bigr).
\end{aligned}
\]
Altogether, choosing $u_{0,j}$ to be smooth enough and bounded away from zero on $\ell_0$ \footnote{Note that in our  all-at-once setting, $u_{0,j}$ does not necessarily need to satisfy a PDE, which (up to closeness to $u^\dagger$) allows for plenty of freedom in its choice.}
\begin{equation}\label{u0}
\tfrac{1}{u_{0,j}(\ell_0)}\in L^\infty(0,L) \,, \quad 
u_{0,j\,y}, \,u_{0,j\,yy}\in W^{1,\infty}(D(\olell)
\,, \quad j\in\{1,2\},
\end{equation}
we have shown that 
\begin{equation}\label{rid}
\|r(\xi)-(\xi-\xi_0)\|_{V^2\times X^\ell\times (X^\gamtil)^2}
\leq C \|\xi-\xi_0\|_{V^2\times X^\ell\times (X^\gamtil)^2}^2
\end{equation}
for some $C>0$. 
Analogously to, e.g., \cite{nonlinearity_imaging_both,nonlinearity_imaging_2d} this provides us with the estimate
\begin{equation}\label{rid_cr}
\begin{aligned}
\exists\, c_r\in(0,1)\, \forall \xi\in U(\subseteq V^2\times X^\ell\times (X^\gamtil)^2)\, : \;&
\|r(\xi)-r(\xi^\dagger)-(\xi-\xi^\dagger)\|_X\\
&\leq c_r\|\xi-\xi^\dagger\|_X
\end{aligned}
\end{equation}
where $\xi^\dagger$ is the actual solution.

\subsubsection*{Linearised uniqueness}
Results on uniqueness of Problem~\ref{Cauchy3} can be found in \cite{Bacchelli:2009,Rundell2008a}. In particular, linear independence of the functions $g_1$, $g_2$ is sufficient for determining both $\ell$ and $\gamtil$ in \eqref{eqn:Bellgamtil}.

Here we will show linearised uniqueness, as this is another ingredient of the convergence proof. More precisely, we show that the intersection of the nullspaces of the linearised forward operator $\Ftil'(u_{0,1},u_{0,2},\ell_0,\gamtil_{0,1},\gamtil_{0,2})$ and the penalisation operator $P$ is trivial.
To this end, assume that $\Ftil'(u_{0,1},u_{0,2},\ell_0,\gamtil_{0,1},\gamtil_{0,2})\, (\du_1,\du_2,\dl,\dgam_1,\dgam_2)=0$ and $P(\du_1,\du_2,\dl,\dgam_1,\dgam_2)=0$, where the latter simply means $\dgam_1=\dgam_2=:\dgam$ and $\dl(0)=0$. 
From the first four lines in \eqref{eqn:Ftilprime} we conclude that $\du_j$ satisfies a homogeneous Cauchy problems and therefore has to vanish on $D(\ell_j)$ for $j\in\{1,2\}$. Thus, $B_{\ell,\gamtil_j}\du_j=0$ and by the same elimination procedure that led to \eqref{eqn:dl}, \eqref{eqn:dgam} (using also $\dl(0)=0$) we obtain that $\dl=0$ and therefore $\dgam=0$.

\subsubsection*{Convergence of Newton type schemes}
\noindent\textbf{(a) Regularised frozen Newton with penalty}\hfill\break
We apply a frozen Newton method with conventional Tikhonov (and no fractional) regularisation but with penalty $P$ as in \cite{nonlinearity_imaging_both,nonlinearity_imaging_2d}.
\begin{equation}\label{frozenNewtonHilbert}
\xi_{n+1}^\delta=\xi_{n}^\delta+(K^\star K+P^\star P+\regpar_n I)^{-1}
\Bigl(K^\star (\vec{h}^\delta-\Ftil(\xi_n^\delta))-P^\star P\xi_n^\delta+\regpar_n(\xi_0-\xi_n^\delta)\Bigr)
\end{equation}
where $K=\Ftil'(\xi_0)$ and $K^\star $ denotes the Hilbert space adjoint of $K:\mathbb{X}\to \mathbb{Y}$ and (cf. \eqref{eqn:spacesHilbert})
\begin{equation}\label{eqn:XY}
\mathbb{X}=H^{3+\epsilon}(D(\olell))^2\times H^1(0,L)\times L^2(0,L)^2,\qquad
\mathbb{Y}:=\bigl(L^2(D(\olell))\times \Gamma_1^2 \times \Gamma_2(\olell) \times L^2(0,L)\bigr)^2
\end{equation} 
The regularisation parameters are taken from a geometric sequence $\regpar_n=\regpar_0\theta^n$ for some $\theta\in(0,1)$, and the stopping index is chosen such that 
\begin{equation}\label{nstar}
\regpar_{n_*(\delta)}\to0\text{ and }\delta^2/\regpar_{n_*(\delta)-1}\to0\text{ as }\delta\to0, 
\end{equation}
where $\delta$ is the noise level according to
\begin{equation}\label{eqn:deltaCauchy3}
\|(f_1^\delta,f_2^\delta,g_1^\delta,g_2^\delta)-(f_1,f_2,g_2,g_2)\|_{L^2(0,T;L^2(\Omega))}\leq\delta.
\end{equation}

An application of \cite[Theorem 2.2]{rangeinvar} together with our verification of range invariance \eqref{rangeinvar_diff} with \eqref{rid} and linearised uniqueness yields the following convergence result.

\begin{theorem}\label{thm:convfrozenNewton}
Let $\xi_0\in U:=\mathcal{B}_\rho(\xi^\dagger)$ for some $\rho>0$ sufficiently small, assume that \eqref{u0} holds and let the stopping index $n_*=n_*(\delta)$ be chosen according to \eqref{nstar}.

Then the iterates $(\xi_n^\delta)_{n\in\{1,\ldots,n_*(\delta)\}}$ are well-defined by \eqref{frozenNewtonHilbert}, remain in $\mathcal{B}_\rho(\xi^\dagger)$ and converge in $X$ 
(defined as in \eqref{eqn:XY}),
$\|\xi_{n_*(\delta)}^\delta-\xi^\dagger\|_X\to0$ as $\delta\to0$. 
\hfill\break
In the noise free case $\delta=0$, $n_*(\delta)=\infty$ we have $\|\xi_n-\xi^\dagger\|_X\to0$ as $n\to\infty$.
\end{theorem}

\medskip
\noindent\textbf{(b) Frozen Newton with penalty applied to fractionally regularised problem}\hfill\break
Replace $\Ftil_j$ in \eqref{eqn:Ftilj} by 
\begin{equation}\label{eqn:Ftilj_alpha}
\Ftil^\alpha_j:(u_{+j},u_{-j},\ell,\gamtil_j)\mapsto
\left(\begin{array}{c}
\partial_y^\alpha \overline{u}^{\olell}_{+j} -\sqrt{-\triangle_x} \overline{u}^{\olell}_{+j} \\
\partial_y u_{-j} -\sqrt{-\triangle_x} u_{-j} \\
u_{+j}\vert_{\Gamma_1}-\tfrac12(f+\sqrt{-\triangle_x}^{\,-1}g)\\
u_{-j}\vert_{\Gamma_1}-\tfrac12(f+\sqrt{-\triangle_x}^{\,-1}g)\\
(Bw_{+j},Bu_{-j})\vert_{\Gamma_2(\olell)}\\
B_{\ell,\gamtil}( u_{+j}+u_{-j})\vert_{\Gamma_0(\ell)}
\end{array}\right)\,,\ j\in\{1,2\}\,,
\end{equation}
where $\partial_y^\alpha w$ is the fractional DC derivative with endpoint $0$ and $\overline{u}^{\olell}(y)=u(\olell-y)$, cf. Section~\ref{sec:Cauchy1}.
Range invariance and linearised uniqueness can be derived analogously to the previous section and therefore we can apply 
\eqref{frozenNewtonHilbert} with $\Ftil^\alpha$ in place of $\Ftil$ and conclude its convergence to a solution
$(u_{+1}^{\alpha,\dagger},u_{-1}^{\alpha,\dagger},u_{+2}^{\alpha,\dagger},u_{-2}^{\alpha,\dagger},\ell^{\alpha,\dagger},\gamtil_1^{\alpha,\dagger}=\gamtil_2^{\alpha,\dagger})$ 
of \eqref{FtilP} with 
$\Ftil^\alpha_j(u_{+j},u_{-j},\ell,\gamtil_j)$ in place of 
$\Ftil_j(u_j,\ell,\gamtil_j)$
for any fixed $\alpha\in(0,1)$.
With the abbreviation
\[
\xi^\alpha := (u_{+,1}^\alpha,u_{-,1}^\alpha,u_{+,2}^\alpha,u_{-,2}^\alpha,\ell^\alpha,\gamtil_1^\alpha=\gamtil_2^\alpha)
\]
we thus have the following convergence result.

\begin{corollary}\label{cor:convfrozenNewton_alpha}
Let $\xi^\alpha_0\in U:=\mathcal{B}_\rho(\xi^{\alpha,\dagger})$ for some $\rho>0$ sufficiently small, assume that \eqref{u0} holds and let the stopping index $n_*=n_*(\delta)$ be chosen according to \eqref{nstar}.

Then the iterates $(\xi_n^{\alpha,\delta})_{n\in\{1,\ldots,n_*(\delta)\}}$ are well-defined by \eqref{frozenNewtonHilbert} with $\Ftil:=\Ftil^{\alpha(\delta)}$, remain in $\mathcal{B}_\rho(\xi^\dagger)$ and converge in $X$ 
(defined as in \eqref{eqn:XY} with $H^{3+\epsilon}(D(\olell))^2$ replaced by $H^{3+\epsilon}(D(\olell))^4$)
$\|\xi_{n_*(\delta)}^{\alpha,\delta}-\xi^{\alpha,\dagger}\|_X\to0$ as $\delta\to0$. In the noise free case $\delta=0$, $n_*(\delta)=\infty$ we have $\|\xi_n^\alpha-\xi^{\alpha,\dagger}\|_X\to0$ as $n\to\infty$.
\end{corollary}

\medskip

It remains to estimate the approximation error $(\ell^{\alpha,\dagger}-\ell^\dagger,\gamtil^{\alpha,\dagger}-\gamtil^\dagger)$.
In Section~\ref{sec:Cauchy1} we have seen that $u_j^\dagger=u_{+,j}^{1,\dagger}+u_{-,j}^{1,\dagger}$ where $u_{-,j}^{\alpha,\dagger}-u_{-,j}^{1,\dagger}=0$ and 
\[
u_{+,j}^{\alpha,\dagger}(x,y)-u_{+,j}^{1,\dagger}(x,y)
=\tfrac12 \sum_{i=1}^\infty (f_{ij}+\tfrac{1}{\sqrt{\lambda_i}}g_{ij})\Bigl(\tfrac{1}{E_{\alpha,1}(-\sqrt{\lambda_i}y^\alpha)}-\exp(\sqrt{\lambda_i}y)\Bigr) \phi_i(x).
\]
Moreover, subtracting the two identities
$B_{\ell^{\alpha,\dagger},\gamtil^{\alpha,\dagger}}u^{\alpha,\dagger}_{+,j}=0$, 
$B_{\ell^{\dagger},\gamtil^{\dagger}}u^{\dagger}_{+,j}=0$, 
we arrive at the following differential algebraic system for $\dl=\ell^{\alpha,\dagger}-\ell^{\dagger}$, $\dgam=\gamtil^{\alpha,\dagger}-\gamtil^{\dagger}$ 
\[
\begin{aligned}
-u^\dagger_{+,1,x}(\ell^\dagger)\,\dl'+d^\alpha_1\,\dl+u^\dagger_{+,1}(\ell^\dagger)\,\dgam=&b^\alpha_1\\
-u^\dagger_{+,2,x}(\ell^\dagger)\,\dl'+d^\alpha_2\,\dl+u^\dagger_{+,2}(\ell^\dagger)\,\dgam=&b^\alpha_2,
\end{aligned}
\]
where 
\[
\begin{aligned}
d^\alpha_j(x)=& \int_0^1 
I(u^{\alpha,\dagger}_{+,j,y},\ell^{\alpha,\dagger},\gamtil^{\alpha,\dagger};x,\ell^\dagger(x)+\theta(\ell^{\alpha,\dagger}(x)-\ell^\dagger(x)))\,d\theta\\
b^\alpha_j(x)=&
-I(u^{\alpha,\dagger}_{+,j}-u^\dagger_{+,j},\ell^{\alpha,\dagger},\gamtil^{\alpha,\dagger};x,\ell^{\alpha,\dagger}(x))\\
I(u,\ell,\gamtil;x,y)=& u_y(x,y)-\ell'(x)u_x(x,y)+\gamtil u(x,y) .
\end{aligned}
\]
Assuming that the Wronskian 
\begin{equation}\label{eqn:W}
W:= u^\dagger_{+,1,x}(\ell^\dagger)\,u^\dagger_{+,2}(\ell^\dagger)-u^\dagger_{+,2,x}(\ell^\dagger)\,u^\dagger_{+,1}(\ell^\dagger) 
\end{equation}
and one of the factors $u^\dagger_{+,j}(\ell^\dagger)$ of $\dgam$ are bounded away from zero, we can conclude existence of a constant $C>0$ independent of $\alpha$ (note that $W$ and $u^\dagger_{+,j}$ do not depend on $\alpha$)
such that
\[
\|\ell^{\alpha,\dagger}-\ell^{\dagger}\|_{C^1(0,L)} 
+ \|\gamtil^{\alpha,\dagger}-\gamtil^{\dagger}\|_{C(0,L)}
\leq C \sum_{j=1}^2\|b^\alpha_j\|_{C(0,L)}
\leq \tilde{C}\sum_{j=1}^2\|u^{\alpha,\dagger}_{+,j}-u^\dagger_{+,j}\|_{C((0,L)\times(\ulell,\olell))}
\]
The latter can be estimated by means of Lemma 1.4 
in Section~\ref{sec:Cauchy1}. 

Combining this with Corollary~\ref{cor:convfrozenNewton_alpha} we have the following convergence result.
\begin{theorem}\label{thm:convfrozenNewtonplus regbyalpha}
Assume that $W$ according to \eqref{eqn:W} and either $u^\dagger_{+,1}(\ell^\dagger)$ or $u^\dagger_{+,2}(\ell^\dagger)$ be bounded away from zero and let $\alpha(\delta)\to0$ as $\delta\to0$.
Moreover, let the assumptions of Corollary~\ref{cor:convfrozenNewton_alpha} hold for $\alpha=\alpha(\delta)$, and let  $(\xi_n^{\alpha(\delta),\delta})_{n\in\{1,\ldots,n_*(\delta)\}}$ be defined by \eqref{frozenNewtonHilbert}, \eqref{nstar} with $\Ftil:=\Ftil^{\alpha(\delta)}$. 
Then
\[
\|\ell_{n_*(\delta)}^{\alpha(\delta),\delta}-\ell^\dagger\|_{H^1(0,L)} + 
\|\gamtil_{n_*(\delta)}^{\alpha(\delta),\delta}-\gamtil^\dagger\|_{L^2(0,L)}
\to0
\text{ as }\delta\to0.
\]
\end{theorem}

\appendix

\section{Appendix: Cauchy 1}\label{sec:appendixCauchy1}

\begin{lemma}\label{lem:estEalpha}
For any $\llam_1\geq0$, $\alpha\in(0,1)$, $p\in(1,\frac{1}{1-\alpha})$, $\hat{p}:= 1+\frac{1}{p}$ and the constant 
\begin{equation}\label{eqn:CalphapTT}
\begin{aligned}
C(\alpha,p,\TT):= \| E_{\alpha,\alpha/2}\|_{L^\infty(\mathbb{R}^+)} 
(\tilde{C}_0(\alpha,\alpha/2)+\tilde{C}_1(\alpha,\alpha/2))
{\textstyle \left(\frac{\hat{p}}{\alpha-\hat{p}}\right)^{\hat{p}}
\max\{1,\sqrt{\TT}\}^{\alpha-\hat{p}}
},
\end{aligned}
\end{equation}
with $\tilde{C}_0$, $\tilde{C}_1$ as in \eqref{eqn:C0C1},
the following estimate holds
\[
\sup_{\llam>\llam_1}\|E_{\alpha,1}(-\llam \, \cdot^\alpha)-\,\cdot^{\alpha-1}E_{\alpha,\alpha}(-\llam \, \cdot^\alpha)\|_{L^p(0,\TT)}\leq C(\alpha,p,\TT)\, (1-\alpha).
\]
\end{lemma}

\begin{proof}
Abbreviating 
\[
\mathfrak{g}_\mu(\ttt)=\tfrac{1}{\Gamma(\mu)} \ttt^{\mu-1}, \quad
e_{\alpha,\beta}(\ttt)=e_{\alpha,\beta}(\ttt;\llam)=\ttt^{\beta-1}E_{\alpha,\beta}(-\llam \ttt^\alpha), \quad 
\] 
the quantity to be estimated is 
$w(\ttt):=E_{\alpha,1}(-\llam  \ttt^\alpha)-\ttt^{\alpha-1}E_{\alpha,\alpha}(-\llam \ttt^\alpha)=e_{\alpha,1}(\ttt)-e_{\alpha,\alpha}(\ttt)$.
Using the Laplace transform identities
\begin{equation}\label{eqn:Laptransf}
(\mathcal{L}\mathfrak{g}_\mu)(\xi)=\xi^{-\mu}, \quad
(\mathcal{L}e_{\alpha,\beta})(\xi)=\frac{\xi^{\alpha-\beta}}{\llam+\xi^\alpha},
\end{equation}
we obtain, for some $\beta\in(\alpha-1,\alpha)$ yet to be chosen, 
\[
(\mathcal{L}w)(\xi)= \frac{\xi^{\alpha-1}-1}{\llam+\xi^\alpha}=\frac{\xi^{\alpha-\beta}}{\llam+\xi^\alpha}(\xi^{\beta-1}-\xi^{\beta-\alpha})
= (\mathcal{L}e_{\alpha,\beta})(\xi)\, \Bigl((\mathcal{L}\mathfrak{g}_{1-\beta})(\xi)-(\mathcal{L}\mathfrak{g}_{\alpha-\beta})(\xi)\Bigr),
\] 
hence, by Young's Convolution Inequality,
\[
\|w\|_{L^p(0,\TT)}=
\|e_{\alpha,\beta}*(\mathfrak{g}_{1-\beta}-\mathfrak{g}_{\alpha-\beta})\|_{L^p(0,\TT)}\\
\leq
\|e_{\alpha,\beta}\|_{L^q(0,\TT)}
\|\mathfrak{g}_{1-\beta}-\mathfrak{g}_{\alpha-\beta}\|_{L^r(0,\TT)}
\]
provided
$\frac{1}{q}+\frac{1}{r}= 1+\frac{1}{p}$.

For the first factor, 
under the condition 
\begin{equation}\label{eqn:condq}
1\leq q<\frac{1}{1-\beta}, 
\end{equation}
that is necessary for integrability near zero, we get
\begin{equation}\label{este}
\|e_{\alpha,\beta}\|_{L^q(0,\TT)} 
\leq 
\|E_{\alpha,\beta}\|_{L^\infty(\mathbb{R}^+)}\frac{\max\{1,\TT\}^{(\beta-1)+1/q}}{((\beta-1)q+1)^{1/q}}.
\end{equation}

The second factor can be estimated by applying the Mean Value Theorem to the function $\theta(\ttt;\alpha,\beta):=\mathfrak{g}_{\alpha-\beta}(\ttt)$ as follows
\[
\begin{aligned}
&\mathfrak{g}_{1-\beta}(\ttt)-\mathfrak{g}_{\alpha-\beta}(\ttt)=\theta(\ttt,1)-\theta(\ttt;\alpha)
=\frac{d}{d\alpha} \theta(\ttt;\tilde{\alpha},\beta)\, (1-\alpha)
=\tilde{\theta}(\ttt;\tilde{\alpha},\beta)\, (1-\alpha)\\
\end{aligned}
\]
where
\[
\begin{aligned}
\tilde{\theta}(\ttt;\tilde{\alpha},\beta)
=& \ttt^{\tilde{\alpha}-\beta-1} \Bigl(\tfrac{1}{\Gamma}(\tilde{\alpha}-\beta)\log(\ttt)-\tfrac{\Gamma'}{\Gamma^2}(\tilde{\alpha}-\beta) \Bigr)\, \\
=&\mathfrak{g}_{\tilde{\alpha}-\beta}(\ttt)\Bigl(\log(\ttt)-(\log\Gamma)'(\tilde{\alpha}-\beta)\Bigr)\,
\end{aligned}
\]
for some $\tilde{\alpha}\in(\alpha,1)$, with the digamma function $\psi=\frac{\Gamma'}{\Gamma}=(\log\circ\Gamma)'$, for which $\frac{\psi}{\Gamma}=\tfrac{\Gamma'}{\Gamma^2}$ is known to be an entire function as is also the reciprocal Gamma function, thus 
\begin{equation}\label{eqn:C0C1}
\tilde{C}_0(\alpha,\beta):=\sup_{\tilde{\alpha}\in(\alpha,1)} |\tfrac{\psi}{\Gamma}(\tilde{\alpha}-\beta)|<\infty, \qquad 
\tilde{C}_1(\alpha,\beta):=\sup_{\tilde{\alpha}\in(\alpha,1)} |\tfrac{1}{\Gamma}(\tilde{\alpha}-\beta)|<\infty.
\end{equation}
Integrability near $\ttt=0$ of $(\ttt^{\tilde{\alpha}-\beta-1})^r$ and of $(\ttt^{\tilde{\alpha}-\beta-1} \log(\ttt))^r$ holds iff 
\begin{equation}\label{eqn:condr}
1\leq r<\frac{1}{1-\tilde{\alpha}+\beta}
\end{equation}
and 
yields 
\begin{equation}\label{estdiffg}
\|\mathfrak{g}_{1-\beta}-\mathfrak{g}_{\alpha-\beta}\|_{L^r(0,\TT)}\leq 
\sup_{\tilde{\alpha}\in(\alpha,1)} \|\tilde{\theta}(\cdot,\tilde{\alpha},\beta)\|_{L^r(0,\ell)}\, (1-\alpha) 
\end{equation}
where 
\begin{equation}\label{estthetatil}
\|\tilde{\theta}(\cdot,\tilde{\alpha},\beta)\|_{L^r(0,\ell)}\leq (\tilde{C}_0(\alpha,\beta) +\tilde{C}_1(\alpha,\beta))
\frac{\max\{1,\TT\}^{(\alpha-\beta-1)+1/r}}{((\alpha-\beta-1)r+1)^{1/r}}.
\end{equation}

Conditions \eqref{eqn:condq}, \eqref{eqn:condr} together with $\tilde{\alpha}\in(\alpha,1)$ are equivalent to  
\[
\frac{1}{p}=\frac{1}{q}+\frac{1}{r}-1
>
1-\tilde{\alpha} 
\text{ and }
1<\beta+\frac{1}{q}<\tilde{\alpha}+\frac{1}{p}
\]
which together with $\tilde{\alpha}\in(\alpha,1)$ leads to the assumption 
\[
\frac{1}{p}>1-\alpha 
\]
and the choice $\min\{0,1-\frac{1}{q}\}<\beta<\alpha+\frac{1}{p}-\frac{1}{q}$.

To minimize the factors\footnote{we do not go for asymptotics with respect to $\TT$ since we assume $\TT$ to be moderately sized anyway} 
\[
c_1(q,\beta)= ((\beta-1)q+1)^{-1/q}, \quad 
c_2(r,\alpha-\beta)= ((\alpha-\beta-1)r+1)^{-1/r}
\]
in \eqref{este}, \eqref{estthetatil}
under the constraints \eqref{eqn:condq}, \eqref{eqn:condr}
and $\frac{1}{q}+\frac{1}{r}= 1+\frac{1}{p}$ we make the choice $\frac{1}{q}=\frac{1}{r}=\frac12(1+\frac{1}{p})$, $\beta=\alpha/2$ that balances the competing pairs $q\leftrightarrow r$, $\beta\leftrightarrow\alpha-\beta$ and arrive at 
\[
c_1(q,\beta)= c_2(r,\alpha-\beta)= \left(\frac{1+\frac{1}{p}}{\alpha-1+\tfrac{1}{p}}\right)^{1+\frac{1}{p}}
\]
\end{proof}

\begin{proof} (Lemma~\ref{lem:rateE})
To prove \eqref{eqn:rateE}, we employ an energy estimate for the {\sc ode} satisfied by $v(\ttt):=E_{\alpha,1}(-\llam \ttt^\alpha)-\exp(-\llam \ttt)=u_{\alpha,\llam}(\ttt)-u_{1,\llam}(\ttt)$,
\[
\begin{aligned}
&\partial_\ttt v+\llam v=-(\partial_\ttt^\alpha-\partial_\ttt)u_{\alpha,\llam}=:\llam w\\
&\text{where }
w=-\tfrac{1}{\llam}(\partial_\ttt^\alpha-\partial_\ttt) E_{\alpha,1}(-\llam \ttt^\alpha)
=E_{\alpha,1}(-\llam \ttt^\alpha)-\ttt^{\alpha-1}E_{\alpha,\alpha}(-\llam \ttt^\alpha)
\end{aligned}
\]
Testing with $|v(\tau)|^{p-1}\mbox{sign}(v(\tau))$, integrating from $0$ to $t$, and applying H\"older's and Young's inequalities yields, after multiplication with $p$,
\begin{equation}\label{vestp}
|v(\ttt)|^p +\llam\int_0^\ttt |v(\tau)|^p\, d\tau \leq \llam\int_0^\ttt |w(\tau)|^p\, d\tau\,,
\end{equation}
in particular
\[
\|v\|_{L^\infty(0,\TT)} \leq \llam^{1/p} \|w\|_{L^p(0,\TT)},\quad
\|v\|_{L^p(0,\TT)} \leq  \|w\|_{L^p(0,\TT)}
\]
for any $\TT\in(0,\infty]$.
The result then follows from Lemma~\ref{lem:estEalpha}.
\end{proof}

\begin{proof} (Lemma~\ref{lem:rateEbounds})
For the second estimate, with $v=e_{\alpha,1}-e_{1,1}$ as in the proof of Lemma~\ref{lem:rateE}, we have to bound $\partial_\ttt v=-\llam(e_{\alpha,\alpha}-e_{1,1})$, where 
\[
\mathcal{L}(e_{\alpha,\alpha}-e_{1,1})(\xi)=\frac{1}{\llam+\xi^\alpha}-\frac{1}{\llam+\xi}
= \frac{\xi-\xi^\alpha}{(\llam+\xi^\alpha)(\llam+\xi)}
= \frac{\xi^{1-\gamma}}{\llam+\xi}\, \frac{\xi^{\alpha-\beta}}{\llam+\xi^\alpha}\,
( \xi^{\beta+\gamma-\alpha} - \xi^{\beta+\gamma-1})
\]
for $\beta,\gamma>0$ with $\beta+\gamma<\alpha$ yet to be chosen.
In view of \eqref{eqn:Laptransf} we thus have
\[
e_{\alpha,\alpha}-e_{1,1} = e_{1,\gamma}*e_{\alpha,\beta}*(\mathfrak{g}_{\alpha-\beta-\gamma}-\mathfrak{g}_{1-\beta-\gamma}) =:e_{1,\gamma}*e_{\alpha,\beta}*\underline{d\mathfrak{g}}
\]
Now, applying the elementary estimate
\[
\begin{aligned}
\|a*b\|_{L^\infty(\ell_2,\ell_3)}
=&\sup_{y\in(\ell_2,\ell_3)}\left|\int_0^y a(y-z) b(z)\, dz\right|\\
=&\sup_{y\in(\ell_2,\ell_3)}\left|\int_0^{y-\ell_1} a(y-z) b(z)\, dz
+\int_{y-\ell_1}^y a(y-z) b(z)\, dz\right|\\
\leq& 
\|a\|_{L^\infty(\ell_1,\ell_3)} \|b\|_{L^1(0,\ell_3-\ell_1)}
+\|b\|_{L^\infty(\ell_2-\ell_1,\ell_3)} \|a\|_{L^1(0,\ell_1)}
\end{aligned}
\]
for $0<\ell_1<\ell_2<\ell_3$, $a,b\in L^1(0,\ell_3)$, 
$a\vert_{(\ell_1,\ell_3)}\in L^\infty(\ell_1,\ell_3)$, 
$b\vert_{(\ell_2-\ell_1,\ell_3)}\in L^\infty(\ell_2-\ell_1,\ell_3)$, twice, namely with 
$a=e_{1,\gamma}$, $b=e_{\alpha,\beta}$, $\ell_1=\ulell/3$, $\ell_2=\ulell$, $\ell_3=\olell$ 
and with 
$a=e_{\alpha,\beta}$, $b=\underline{d\mathfrak{g}}$, $\ell_1=\ulell/3$, $\ell_2=2\ulell/3$, $\ell_3=\olell$,
along with Young's Convolution Inequality
we obtain
\[
\begin{aligned}
\|e_{1,\gamma}*&(e_{\alpha,\beta}*\underline{d\mathfrak{g}})\|_{L^\infty(\ulell,\olell)}\\
&\leq
\|e_{1,\gamma}\|_{L^\infty(\ulell/3,\olell)} \|e_{\alpha,\beta}*\underline{d\mathfrak{g}}\|_{L^1(0,\olell)}
+\|e_{\alpha,\beta}*\underline{d\mathfrak{g}}\|_{L^\infty(2\ulell/3,\olell)} \|e_{1,\gamma}\|_{L^1(0,\olell)}
\\
&\leq
\|e_{1,\gamma}\|_{L^\infty(\ulell/3,\olell)} \|e_{\alpha,\beta}\|_{L^1(0,\olell)}
\|\underline{d\mathfrak{g}}\|_{L^1(0,\olell)}\\
&\quad +
\Bigl(
\|e_{\alpha,\beta}\|_{L^\infty(\ulell/3,\olell)} \|\underline{d\mathfrak{g}}\|_{L^1(0,\olell-\ulell/3)}
+\|\underline{d\mathfrak{g}}\|_{L^\infty(\ulell/3,\olell)} \|e_{\alpha,\beta}\|_{L^1(0,\ulell/3)}
\Bigr)
\|e_{1,\gamma}\|_{L^1(0,\olell)}
\end{aligned}
\]
Using this with $\beta=\gamma=\alpha/3$ and (cf. \eqref{estdiffg}) 
\[
\begin{aligned}
\|\underline{d\mathfrak{g}}\|_{L^1(0,\olell)}&\leq 
\sup_{\tilde{\alpha}\in(\alpha,1)} \|\tilde{\theta}(\cdot,\tilde{\alpha},\alpha/3)\|_{L^1(0,\olell)}\, (1-\alpha), \\
\|\underline{d\mathfrak{g}}\|_{L^\infty(\ulell/3,\olell)}&\leq 
\sup_{\tilde{\alpha}\in(\alpha,1)} \|\tilde{\theta}(\cdot,\tilde{\alpha},\alpha/3)\|_{L^\infty(\ulell/3,\olell,\olell)}\, (1-\alpha),
\end{aligned}
\]
we arrive at the second estimate in \eqref{eqn:rateEbounds} with 
$\hat{C}(\alpha_0,\ulell,\olell)=\sup_{\alpha\in(\alpha_0,1)} \check{C}(\alpha,\ulell,\olell)$,
\[
\begin{aligned}
\check{C}(\alpha,\ulell,\olell)=\,&
\|e_{1,\alpha/3}\|_{L^\infty(\ulell/3,\olell)} \|e_{\alpha,\alpha/3}\|_{L^1(0,\olell)}
\sup_{\tilde{\alpha}\in(\alpha,1)} \|\tilde{\theta}(\cdot,\tilde{\alpha},\alpha/3)\|_{L^1(0,\olell)}
\\&+ 
\|e_{1,\alpha/3}\|_{L^1(0,\olell)} \|e_{\alpha,\alpha/3}\|_{L^\infty(\ulell/3,\olell)}
\sup_{\tilde{\alpha}\in(\alpha,1)} \|\tilde{\theta}(\cdot,\tilde{\alpha},\alpha/3)\|_{L^1(0,\olell)}
\\&+ 
\|e_{1,\alpha/3}\|_{L^1(0,\olell)} \|e_{\alpha,\alpha/3}\|_{L^1(0,\ulell/3)}
\sup_{\tilde{\alpha}\in(\alpha,1)} \|\tilde{\theta}(\cdot,\tilde{\alpha},\alpha/3)\|_{L^\infty(\ulell/3,\olell,\olell)}
\end{aligned}
\]

The first estimate can be shown analogously.

\end{proof}

\section{Appendix: Cauchy 2}\label{sec:appendixCauchy2}

In the impedance case, using the PDE, the right hand side of \eqref{FprimeCauchy2} can be written as 
\[
\begin{aligned}
G(u,\ell)\dl
=& \dl'(x)\Bigl(u_x(x,\ell(x))-\tfrac{\ell'(x)}{\sqrt{1+\ell'(x)^2}}\gamma(x)u(x,\ell(x))\Bigr)\\
&+\dl (x)\Bigl(\tfrac{d}{dx}\Bigl[u_x(x,\ell(x))\Bigr]
-\sqrt{1+\ell'(x)^2}\gamma(x)u_y(x,\ell(x))\Bigr)\\
=& \tfrac{d}{dx}\Bigl[\dl (x)\Bigl(u_x(x,\ell(x))-\tfrac{\ell'(x)}{\sqrt{1+\ell'(x)^2}}\gamma(x)u(x,\ell(x))\Bigr)\Bigr]\\
&+\dl (x)\Bigl(-\sqrt{1+\ell'(x)^2}\gamma(x)u_y(x,\ell(x))+\tfrac{d}{dx}\Bigl[\tfrac{\ell'(x)}{\sqrt{1+\ell'(x)^2}}\gamma(x)u(x,\ell(x))\Bigr]\Bigr)\\
=&\tfrac{d}{dx}\phi(x)-a(x)\phi(x)\ = \tfrac{d}{dx}[\dl (x)\coeffalpha[\ell,u](x)]-\dl (x)\coeffbeta[\ell,u](x)
\end{aligned}
\]
Here using the impedance conditions on $u$, that yield
\begin{equation}\label{eqn:alpha}
\begin{aligned}
&u_x(x,\ell(x))-\tfrac{\ell'(x)}{\sqrt{1+\ell'(x)^2}}\gamma(x)u(x,\ell(x))\\
&=u_x(x,\ell(x))+\tfrac{\ell'(x)}{1+\ell'(x)^2}\bigl(u_y(x,\ell(x))-\ell'(x)u_x(x,\ell(x))\bigr)\\
&=\tfrac{1}{1+\ell'(x)^2}\bigl(u_x(x,\ell(x))+\ell'(x)u_y(x,\ell(x))\bigr)
=\tfrac{\partial_{\ttau}u(x,\ell(x))}{1+\ell'(x)^2}
=\tfrac{1}{1+\ell'(x)^2}\tfrac{d}{dx} u(x,\ell(x))\\
&=
\begin{cases}
\tfrac{d}{dx} u(x,\ell(x))=u_x(x,\ell(x)) &\text{ if } \ell'(x)=0\\
\tfrac{1}{\ell'(x)}\Bigl(\tfrac{\gamma(x)}{\sqrt{1+\ell'(x)^2}}u(x,\ell(x))+u_y(x,\ell(x)))\Bigr)
&\text{ else}\end{cases}
\qquad=:\coeffalpha[\ell,u](x)
\end{aligned}
\end{equation}
In our implementation we use the last expression that is based on 
\[
u_x(x,\ell(x))=
\frac{1}{\ell'(x)}\Bigl(\sqrt{1+\ell'(x)^2}\gamma(x)u(x,\ell(x))+u_y(x,\ell(x)))\Bigr)
\quad \text{ if }\ell'(x)\not=0,
\]
since $u_x$ is difficult to evaluate numerically unless the boundary is flat (case $\ell'(x)=0$). 
Moreover, 
\begin{equation}\label{eqn:beta}
\begin{aligned}
&\Bigl(-\sqrt{1+\ell'(x)^2}\gamma(x)u_y(x,\ell(x))+\tfrac{d}{dx}\Bigl[\tfrac{\ell'(x)}{\sqrt{1+\ell'(x)^2}}\gamma(x)u(x,\ell(x))\Bigr]\Bigr)\\
&=\Bigl\{\tfrac{\ell''(x)}{\sqrt{1+\ell'(x)^2}^3}\gamma(x)
+\tfrac{\ell'(x)}{\sqrt{1+\ell'(x)^2}} \gamma'(x) +\gamma(x)^2\Bigr\} u(x,\ell(x))
\qquad=:\coeffbeta[\ell,u](x)
\end{aligned}
\end{equation}
we have 
\begin{equation}\label{phia}
\phi(x)= \dl(x)\, \coeffalpha[\ell,u](x)
\qquad
a(x)=
\tfrac{\coeffbeta[\ell,u]}{\coeffalpha[\ell,u]}(x)
\end{equation}
with $\coeffalpha$, $\coeffbeta$ as defined in \eqref{eqn:alpha}, \eqref{eqn:beta}.
\\
Thus the Newton step in the impedance case reads as 
\begin{equation*}
\begin{aligned}
(I)\ \ell^{(k+1)}(x)&=\ell^{(k)}(x)-\tfrac{1}{\coeffalpha[\ell^{(k)},u^{(k)}](x)}
\Bigl\{
\exp\Bigl(-\int_0^x\tfrac{\coeffbeta[\ell^{(k)},u^{(k)}]}{\coeffalpha[\ell^{(k)},u^{(k)}]}(s)\, ds\Bigr) \dl(0)\coeffalpha[\ell^{(k)},u^{(k)}](0)\\
&\hspace*{5cm}+\int_0^x b(s) \exp\Bigl(-\int_s^x\tfrac{\coeffbeta[\ell^{(k)},u^{(k)}]}{\coeffalpha[\ell^{(k)},u^{(k)}]}(t)\, dt\Bigr)\, ds
\Bigr\}\\
&\mbox{where } 
b(x)=\partial_\nnu \bar{z}(x,\ell^{(k)}(x)) + \sqrt{1+\ell'(x)^2}\gamma(x)\bar{z}(x,\ell^{(k)}(x)).
\end{aligned}
\end{equation*}
In particular, with Neumann conditions on the lateral boundary $B=\pm\partial_x$ under the compatibility condition $\ell'(0)=0$ we have $\coeffalpha[\ell,u](0)=\partial_\ttau(0,\ell(0))=u_x(0,\ell(0))=0$ and therefore 
\[
(I)\ \ell^{(k+1)}(x)=\ell^{(k)}(x)-\tfrac{1}{\coeffalpha[\ell^{(k)},u^{(k)}](x)}\int_0^x b(s) \exp\Bigl(-\int_s^x\tfrac{\coeffbeta[\ell^{(k)},u^{(k)}]}{\coeffalpha[\ell^{(k)},u^{(k)}]}(t)\, dt\Bigr)\, ds
= \ell^{(k)}(x)-\dl(x),
\]
where the value at the left hand boundary point can be computed by means of l'Hospital's rule as (skipping the argument $[\ell^{(k)}u^{(k)}]$ for better readability)
\[
\lim_{x\to0}\dl(x) = \lim_{x\to0}\frac{\phi(x)}{\coeffalpha(x)} 
= \lim_{x\to0}\frac{\phi'(x)}{\coeffalpha'(x)}
= \lim_{x\to0}\frac{b(x)-\tfrac{\coeffbeta}{\coeffalpha}(x)\phi(x)}{\coeffalpha'(x)}
= \lim_{x\to0}\frac{b(x)-\coeffbeta(x) \dl(x)}{\coeffalpha'(x)}
\]
hence
\[
\lim_{x\to0}\dl(x) 
=\lim_{x\to0} \frac{1}{1+\frac{\coeffbeta(x)}{\coeffalpha'(x)}} \ \frac{b(x)}{\coeffalpha'(x)}
= \lim_{x\to0}\tfrac{b(x)}{\coeffalpha'(x)+\coeffbeta(x)}
= \tfrac{\bar{z}_y(0,\ell^{(k)}(0))+\gamma(0)\bar{z}(0,\ell^{(k)}(0))}{
u^{(k)}_{xx}(0,\ell^{(k)}(0))+({\ell^{(k)}}''\gamma(0)+\gamma(0)^2)u^{(k)}(0,\ell^{(k)}(0))}.
\]

\section*{Acknowledgement}
The work of the first author was supported by the Austrian Science Fund through grant P36318;
the second author was supported in part by the National Science Foundation through award DMS-2111020.

\end{document}